\documentclass[12pt]{amsart}
\usepackage{pifont}
\usepackage{txfonts}
\usepackage{}
\usepackage{amsfonts}
\usepackage{graphicx}
\usepackage{epstopdf}
\usepackage{amsmath}
\usepackage{amsthm}
\usepackage{latexsym,bm}
\usepackage{amssymb}
\usepackage{esint}
\usepackage{enumitem}
\usepackage{indentfirst}
\usepackage{amscd}
\newtheorem{thm}{Theorem}[section]
\newtheorem{cor}[thm]{Corollary}
\newtheorem{lem}[thm]{Lemma}

\newtheorem{Remark}{Remark}
\numberwithin{equation}{section}
\numberwithin{Remark}{section}

\begin{document}

\title{A Loewner-Nirenberg phenomenon for Ricci flow on compact manifolds with boundary.II}

\author{Gang Li$^\dag$}

\begin{abstract}
This is a continuation of the research in \cite{Li4}. Let $(\overline{M},g_{-1})$ be a closed geodesic $r_0$-ball in the hyperbolic space $(\mathbb{H}^n,g_{-1})$. Let $m\neq1$ be a positive constant. In this paper, we show that for $n\geq3$, starting from the metric $m g_{-1}$ on $\overline{M}$, with certain prescribed non-decreasing rotationally symmetric mean curvature and the fixed conformal class $[g_{\mathbb{S}^{n-1}}]$ on the boundary $\partial M$, the solution $g(t)$ to the normalized Ricci flow $(\ref{equn_NRF1})$ which is continuous up to the boundary, exists for all $t>0$, and converges locally uniformly in the interior $M$ of $\overline{M}$ to a complete hyperbolic metric as $t\to\infty$(see Theorem \ref{thm_main} for details). Under some additional conditions, we show the same conclusion holds for $n=2$.
\end{abstract}

\renewcommand{\subjclassname}{\textup{2000} Mathematics Subject Classification}
 \subjclass[2010]{Primary 53C44; Secondary 35K55, 35R01, 53C21}


\thanks{$^\dag$ Research supported by the National Natural Science Foundation of China No. 11701326 and Shandong Provincial Natural Science Foundation No. ZR2026MS0006.}

\address{Gang Li, School of Mathematics, Shandong University, Jinan, Shandong Province, China}
\email{runxing3@gmail.com}

\maketitle


\section{Introduction}


Let $\overline{M}=\overline{B_{r_0}}(0)$ be the closed $r_0$-ball on the Euclidean space $\mathbb{R}^n$, for some $r_0>0$ and $n\geq2$, with $M=B_{r_0}(0)$ its interior and $\partial M$ its boundary. We equip $\overline{M}$ with a rotationally symmetric metric with sectional curvature $-\frac{1}{m}$:
\begin{align}\label{equn_mhyperbolicmetric}
m\,g_{-1}=m\,(dr^2+\sinh^2(r)\,g_{\mathbb{S}^{n-1}})
\end{align}
for $0\leq r\leq r_0$, where $g_{-1}$ is the hyperbolic metric, $m\neq1$ is a positive constant and $g_{\mathbb{S}^{n-1}}$ is the round metric on $\mathbb{S}^{n-1}$, and hence, $(\overline{M},mg_{-1})$ is the closed geodesic $\sqrt{m}r_0$-ball in $(\mathbb{H}^n,mg_{-1})$. In this paper, we consider the following initial-boundary value problem of the normalized Ricci flow
\begin{align}
&\label{equn_NRF1}\frac{\partial}{\partial t}g=-2(\text{Ric}_{g}+(n-1)g),\,\,\,\text{in}\,\,\overline{M}\times[0,\infty),\\
&\label{equn_Nbdc1}H(g)=\eta(t)\,\,\,\text{on}\,\,\partial M\times[0,\infty),\\
&\label{equn_Nbdc2}[g^T]=[g_{S^{n-1}}]\,\,\,\text{on}\,\,\partial M\times[0,\infty),\\
&\label{equn_Ninc2m}g\big|_{t=0}=mg_{-1}\,\,\,\text{in}\,\,\overline{M},
\end{align}
where $H(g)$ is the mean curvature of $\partial M$ with respect to $g=g(t)$, also $[g^T]$ is the conformal class of the tangential part $g^T$ of the metric $g=g(t)$ on $\partial M$, and $\eta=\eta(t)\in C^k([0,\infty))$ with $k\geq 2$. In \cite{Li4}, the author showed that for $m=1$, when $\eta(t)$ is increasing with its derivative $\eta'(t)$ not identically zero and that
\begin{align*}
&\eta(0)=H_{\,g_{-1}}\big|_{\partial M}=(n-1)\frac{\cosh(r_0)}{\sinh(r_0)},\,\,\eta^{(j)}(0)=0,
\end{align*}
for $j=1,..,k$, there exists a unique solution $g=g(t)$ to $(\ref{equn_NRF1})-(\ref{equn_Ninc2m})$ for all $t>0$, and $g(t)$ converges locally uniformly to a complete hyperbolic metric in $M$ as $t\to\infty$.

We recall that under the initial condition $\bar{g}\big|_{t=0}=mg_{-1}$ on $\mathbb{H}^n$, we have a global solution to the normalized Ricci flow $(\ref{equn_NRF1})$ on $\mathbb{H}^n$:
\begin{align}\label{equn_modelg}
\bar{g}(t)=\xi(t)g_{-1}
\end{align}
for $t\geq0$, with
\begin{align}
\xi(t)=1+e^{-2(n-1)t}(m-1).
\end{align}
Here $\xi=\xi(t)$ satisfies the equation
\begin{align}
\xi'(t)=2(n-1)(1-\xi(t)).
\end{align}
Restricted on $\overline{M}$, the mean curvature on the boundary with respect to $\bar{g}(t)$ satisfies
\begin{align}\label{equn_Hbd0}
H_{\bar{g}(t)}\big|_{\partial M}=(n-1)\xi(t)^{-\frac{1}{2}}\frac{\cosh(r_0)}{\sinh(r_0)}
\end{align}
for $t\geq0$. Now we state the main theorem of the paper:
\begin{thm}\label{thm_main}
Let $n\geq3$, $m>0$, $m\neq1$ and $k\geq2$. Assume the function $\rho=\rho(t)\in C^k([0,\infty))$ satisfies that
\begin{align}\label{equn_comptmk}
\rho(0)=\rho^{(j)}(0)=0
\end{align}
for $1\leq j\leq k$, $\rho'(t)\geq0$ for $t\geq0$, and $\rho'(t)>0$ for $t\in(0,t_0)$ with some $t_0>0$. Define
\begin{align}\label{inequn_etabc2}
\eta(t)=H_{\bar{g}(t)}\big|_{\partial M}+\rho(t),
\end{align}
for $t\geq0$, with $\bar{g}$ in $(\ref{equn_modelg})$ and $H_{\bar{g}(t)}\big|_{\partial M}$ in $(\ref{equn_Hbd0})$. If $m>1$, we further assume that
\begin{align}\label{inequn_rhodt0}
\rho'(t)\geq(n-1+\sigma)(1-\xi(t)^{-1})\rho(t)
\end{align}
for $t\in(0,\epsilon)$, where $\sigma$ and $\epsilon$ are some small positive constants, and moreover,
\begin{align}\label{inequn_rhodt2}
\rho'(t)\geq(n-1)(1-\xi(t)^{-1})\rho(t)
\end{align}
for $t\geq0$ when $m>1$. Then the solution $g(t)$ to $(\ref{equn_NRF1})-(\ref{equn_Ninc2m})$ exists for all $t>0$, and  converges locally uniformly in $M$ to the complete hyperbolic metric $g_{\infty}$ on $M$, as $t\to\infty$.
\end{thm}
We remark that when $m>1$, since the conditions $(\ref{inequn_rhodt0})$ and $(\ref{inequn_rhodt2})$ are linear, we can choose the function $\rho(t)$ in Theorem \ref{thm_main} so that $\lim_{t\to\infty}\rho(t)$ could be sufficiently close to $0$. In \cite{Li1} and \cite{Li3}, we consider the convergence, of the solutions to the initial-boundary value problem of certain conformal flows, to the solution to the blow-up boundary value problem of the Yamabe equation (i.e. the Loewner-Nirenberg problem, see \cite{LN,AM,AM1,ACF}). Theorem \ref{thm_main} can be considered as a generalization of this approach to the blow-up boundary value problem of the Einstein equations, which constitutes the first try of the flow approach starting from metrics continuous up to the boundary to the fill-in problem of conformally compact Einstein (CCE) metrics.

Recall that a family of metrics $h(t)$ satisfies the Ricci flow equation
\begin{align*}
h_t=-2\text{Ric}_{h(t)}
\end{align*}
for $t\geq0$ if and only if the family of metrics $g(t)$ defined by
\begin{align}\label{equn_vartransf}
g(t)=e^{-2(n-1)t}\,h(\frac{1}{2(n-1)}(e^{2(n-1)t}-1))
\end{align}
satisfies $(\ref{equn_NRF1})$ for $t\geq0$. Hence, by Theorem 1.2 and Theorem 1.3 in \cite{Gi}, there exists a unique solution to $(\ref{equn_NRF1})-(\ref{equn_Ninc2m})$  of $C^{2k+\alpha,k+\frac{\alpha}{2}}$ on $\overline{M}\times[0,T)$ for any $\alpha\in(0,1)$ and some $T>0$ under the compatibility conditions $(\ref{equn_comptmk})$. Since $g\big|_{t=0}$ and the boundary data are rotationally symmetric, by uniqueness, the solution $g(t)$ has the form
\begin{align}\label{equn_metricpolar1}
g(t)=a(r,t)^2dr^2+b(r,t)^2g_{\mathbb{S}^{n-1}},
\end{align}
for $(r,t)\in [0,r_0]\times[0,T)$, and hence the boundary condition $(\ref{equn_Nbdc1})$ is equivalent to
\begin{align}
\text{II}_{g}=\frac{\eta}{n-1}g^T,
\end{align}
with $\text{II}_g$ the second fundamental form of the boundary, as posed in \cite{Sh,Co1,Pu}. In comparison to the case $m=1$ in \cite{Li4}, here, instead of proving that the sectional curvature of $g(t)$ is bounded from above by $-1$, we show that the sectional curvature of $g(t)$ is bounded from above by $-\xi(t)^{-1}$ for $t\geq0$. After that, by a global argument, we show that the solution $g(t)$ exists for all $t>0$, and the solution converges locally uniformly to a rotationally symmetric locally hyperbolic metric $g_{\infty}$ in $M$. Finally, we show that $g_{\infty}$ is complete in $M$, by proving that $\text{Vol}_{g(t)}(\overline{M})\to\infty$ as $t\to\infty$, which is more complicated than the argument in \cite{Li4}.

For dimension two, due to the vanishing of several good terms in the a priori estimates (see Section \ref{section_6}), we have to employ the tools developed in \cite{Li3} with more restrictive conditions, and we have the following theorem:

\begin{thm}\label{thm_convern2}
Let $m>0$, $m\neq1$ and $\overline{M}=\overline{B_{r_0}}(0)$ be a closed geodesic ball of radius $r_0>0$ on the hyperbolic space $(\mathbb{H}^2,g_{-1})$. Assume the boundary data $\eta$ in $(\ref{equn_Nbdc1d2})$ is defined by
\begin{align*}
\eta(t)=k_{\bar{g}}\big|_{\partial M}+\rho(t)
\end{align*}
for $t\geq0$, with $\bar{g}$ in $(\ref{equn_modelg})$ and
\begin{align}\label{equn_Hbd02D}
k_{\bar{g}(t)}\big|_{\partial M}=\xi(t)^{-\frac{1}{2}}\frac{\cosh(r_0)}{\sinh(r_0)}
\end{align}
and $\rho=\rho(t)\in C^k([0,\infty))$ satisfies the compatibility condition $(\ref{equn_comptmk})$ with $k\geq2$, $\rho'(t)\geq 0$ for $t\geq0$ and $\rho'(t)>0$ for $t\in(0,t_1)$ with some $t_1>0$. Moreover, if $m>1$, we assume that 
$(\ref{inequn_rhodt2})$ holds for $t\geq0$, with $n=2$. Moreover, suppose that
 \begin{align}\label{inequn_etaupperbd}
\eta\leq y(t)^{\frac{1}{3}}-2,
\end{align}
 on $\partial M\times[0,\infty)$, where $y(t)\in C^3([0,\infty))$ is some positive function satisfying
\begin{align}\label{inequn_ODEgrowth}
y'\geq 3y+1
\end{align}
for $t\in[0,\infty)$. Assume that when $m\in(0,1)$ there exists a constant $T>0$ and a small constant $\epsilon>0$ such that
\begin{align}\label{inequn_rhoinfty}
\rho'(t)\geq \epsilon (1+t)^{-1}(\ln(1+t))^{-1}
 \end{align}
 for $t\geq T$; while when $m>1$ there exists a constant $T>0$ and a small constant $\epsilon>0$ such that
\begin{align}\label{inequn_rhoinfty2}
\rho'(t)-(1-\xi(t)^{-1})\rho(t)\geq \epsilon (1+t)^{-1}(\ln(1+t))^{-1}
 \end{align}
 for $t\geq T$. Then the solution $g(t)$ to $(\ref{equn_NRF1d2})-(\ref{equn_Ninc2d2})$ converges locally uniformly to a complete hyperbolic metric in $M$ as $t\to\infty$.
\end{thm}

\begin{Remark}
In the appendix, we use ChatGPT 6 Astra to give a proof of Theorem 1.2 without the conditions $(\ref{inequn_rhoinfty})$ and $(\ref{inequn_rhoinfty2})$ based on the Gauss-Bonnet formula, and hence Corollary \ref{cor_1.3} holds without the conditions $(\ref{inequn_rhoinfty})$ and $(\ref{inequn_rhoinfty2})$ on $\eta(t)$ for $t>0$ large.
\end{Remark}

Here we need the conditions $(\ref{inequn_rhoinfty})$ and $(\ref{inequn_rhoinfty2})$ when we use the variation argument on the volume of $M$ with respect to $g(t)$ to show the completeness of the limit metric $g_{\infty}$ on $M$. As a direct consequence of Theorem \ref{thm_convern2} above, and Theorem 1.3 in \cite{Li3}, we obtain:
\begin{cor}\label{cor_1.3}
Let $m>0$, $m\neq1$ and $\overline{M}=\overline{B_{r_0}}(0)$ be a closed geodesic ball of radius $r_0>0$ on the hyperbolic space $(\mathbb{H}^2,g_{-1})$. Let $\eta=\eta(t)$ satisfy the conditions in Theorem \ref{thm_convern2}. Assume $u_0\in C^{2+\alpha}(\overline{M})$ and $\psi\in C^{1+\alpha,\frac{1}{2}+\frac{\alpha}{2}}(\partial M\times[0,T])$ for all $T>0$, and also the compatibility condition holds on $\partial M\times\{0\}$:
\begin{align}
k_{e^{2u_0}g_{-1}}=\psi(\cdot,0).
\end{align}
Suppose $u_0\geq\frac{1}{2}\ln(m)$ on $\overline{M}$, and $\psi$ satisfies
\begin{align*}
\eta\leq \psi\leq y_1(t)^{\frac{1}{3}}-2
\end{align*}
on $\partial M\times[0,\infty)$, where $y_1(t)\in C^3([0,\infty))$ is some positive function satisfies $(\ref{inequn_ODEgrowth})$ for $t\in[0,\infty)$. Then there exists a unique solution $g(t)$ for all $t>0$ to $(\ref{equn_NRF1})$ with the initial-boundary conditions:
\begin{align*}
&k_{g(t)}=\psi,\,\,\text{on}\,\,\partial M\times[0,\infty),\\
&g(0)=e^{2u_0}g_{-1}\,\,\text{on}\,\,\overline{M}.
\end{align*}
Moreover, $g(t)$ converges locally uniformly to a complete hyperbolic metric in the interior of $\overline{M}$, as $t\to\infty$.
\end{cor}

Throughout this article, the notation $\overline{M}=\overline{B_{r_0}}(0)$ always means the closed $r_0$-ball centered at the origin in the Euclidean space $\mathbb{R}^n$, and the coordinates on $\overline{M}$ will always be expressed using the polar coordinates $(r,\mathbb{S}^{n-1})$ on $\mathbb{R}^n$ with $0\leq r\leq r_0$, and moreover, the starting hyperbolic metric of the Ricci flow on $\overline{M}$ is expressed as $(\ref{equn_mhyperbolicmetric})$ for $0\leq r\leq r_0$. The notation $\overline{B_r}(0)$ always means the $r$-ball centered at the origin in $\mathbb{R}^n$, with $B_r(0)$ its interior and $\partial B_r(0)$ its boundary, for $0\leq r\leq r_0$.

\vskip0.2cm
{\bf Statement on the use of AI.} This article does not use AI; however, we have included the proof of Theorem \ref{thm_convern2}, with the conditions $(\ref{inequn_rhoinfty})$ and $(\ref{inequn_rhoinfty2})$ removed, in the appendix using ChatGPT Astra.
\vskip0.2cm
{\bf Acknowledgements.} The author would like to thank Yuxing Deng, Li Wu, Yunrui Zheng and Tao Tao for helpful discussion. The author would also like to thank Professor Sun-Yung Alice Chang for her encouragement and helpful discussion.

\vskip0.2cm

\section{Preliminaries}

Let $\overline{M}=\overline{B_{r_0}}(0)\subseteq \mathbb{R}^n$, for $n\geq 2$. By the discussion below Theorem \ref{thm_main} in the introduction, there exists a unique solution $g(t)$ to $(\ref{equn_NRF1})-(\ref{equn_Ninc2m})$ on $\overline{M}\times[0,T)$ with $T>0$ the largest existence time, and $g(t)$ is rotationally symmetric and of the following form:
\begin{align}\label{equn_metricpolar}
g=a(r,t)^2dr^2+b(r,t)^2g_{\mathbb{S}^{n-1}},
\end{align}
under the geodesic polar coordinates on $(\overline{M}, g_{-1})$, where for $0\leq r\leq r_0$,
\begin{align}
a(r,0)=\sqrt{m},\,\,\,\,b(r,0)=\sqrt{m}\sinh(r).
\end{align}
By direct computation, the normalized Ricci flow is not a conformal flow for $n\geq3$, unless $g(t)$ has constant sectional curvature for any $t\geq0$, which contradicts our assumptions. Denote $s$ the distance function to the origin under the metric $g(t)$ for $T>t\geq0$, i.e.,
\begin{align}\label{equn_saformul}
s=s(r,t)=\int_0^ra(x,t)dx.
\end{align}
Thus by the regularity of $g$ at the origin,
\begin{align}\label{equn_originr}
\frac{d}{ds}b(0,t)=1,\,\,\,\frac{d^{2k}}{d s^{2k}}b(0,t)=0,
\end{align}
for $T>t\geq0$ and any integer $k\geq0$. Therefore, for any $T>t\geq0$,
\begin{align*}
&\frac{\partial}{\partial r}=a\,\frac{\partial}{\partial s},\\
&g(t)=ds^2+ b^2g_{\mathbb{S}^{n-1}}.
\end{align*}
Direct calculation yields the formula for the Ricci tensor:
\begin{align}
\text{Ric}_g=-(n-1)b^{-1}b_{ss}''a^2dr^2\,-\,b[b_{ss}''-(n-2)b^{-1}(1-(b_s')^2)]\,g_{\mathbb{S}^{n-1}}.
\end{align}
Since $g(t)$ is rotationally symmetric, by adopting the same notations as in \cite{AK,DG}, we denote the sectional curvatures of the $2$-plane perpendicular to the fibers $\{r\}\times \mathbb{S}^{n-1}$ and of the $2$-planes tangential to the fibers by $K$ and $L$ respectively, and hence the curvature of $g$ is entirely described by $K$ and $L$ with the formulas (see \cite{Li4})
 \begin{align}
 K=-b^{-1}b_{ss}'',\,\,\,L=b^{-2}(1-(b_s')^2).
 \end{align}
By $(\ref{equn_originr})$, when $g(t)$ is of $C^{k+2}$ in $\overline{M}$, we obtain that $K(\cdot,t),L(\cdot,t)\in C^k(\overline{M})$. For convenience of later discussion, following \cite{Li4}, we employ the following curvature functions:
\begin{align}
&F=F(r,t)\triangleq-(n-1)b^{-1}b_{ss}''=(n-1)K,\\
&G=G(r,t)\triangleq -b^{-1}b_{ss}''+(n-2)b^{-2}(1-(b_s')^2)=K+(n-2)L,\\
&P=F+n-1,\,\,\,Q=G+n-1,\\
&F_1=F+(n-1)\xi(t)^{-1},\,\,\,\,F_2=G+(n-1)\xi(t)^{-1}.
\end{align}
Therefore, the scalar curvature has the expression
\begin{align*}
R_g=F+(n-1)G.
\end{align*}
From $(\ref{equn_NRF1})$, we derive the evolution equations for $a$ and $b$:
\begin{align}
&\label{equn_NRFa}a_t=(n-1)a(b^{-1}b_{ss}''-1) = -P\,a,\\
&\label{equn_NRFb}b_t=b[b^{-1}b_{ss}''-(n-2)b^{-2}(1-(b_s')^2)-n+1]=-Q\,b,
\end{align}
and hence for $(r,t)\in[0,r_0]\times[0,T)$, we obtain
\begin{align}
&\label{equn_aintt0}a(r,t)=a(r,0)e^{-\int_0^tP(r,\tau)d\tau}=a(r,0)e^{-\int_0^tF_1(r,\tau)+(n-1)(1-\xi(\tau)^{-1})d\tau}=a(r,0)(m^{-1}\xi(t))^{\frac{1}{2}}e^{-\int_0^tF_1(r,\tau)d\tau},\\
&\label{equn_bintt0}b(r,t)=b(r,0)e^{-\int_0^tQ(r,\tau)d\tau}=b(r,0)e^{-\int_0^tF_2(r,\tau)+(n-1)(1-\xi(\tau)^{-1})d\tau}=b(r,0)(m^{-1}\xi(t))^{\frac{1}{2}}e^{-\int_0^tF_2(r,\tau)d\tau}.
\end{align}
We now recall the following lemma in \cite{Li4} about the convexity of the geodesic spheres $\partial B_{r}(0)$.
\begin{lem}\label{lem_Hp} (Lemma 2.1 in \cite{Li4})
If $\eta(t)>0$ for $t\in[0,T)$, then $b_r'(r,t)>0$ for $(r,t)\in[0,r_0]\times[0,T)$.
\end{lem}

For any $T>t\geq0$, using the Taylor expansion of $b$ near $r=0$ (i.e., $s=0$),
\begin{align*}
b(r,t)=s+\frac{s^3}{3!}\partial_s^3b(0,t)+\frac{s^5}{5!}\partial_s^5b(0,t)+O(s^7),
\end{align*}
we get
\begin{align}
F_2-F_1=Q-P=(n-2)b^{-2}[1-(b_s')^2]+(n-2)b^{-1}b_{ss}''=O(s^2).
\end{align}
In particular,
\begin{align}
&P(0,t)=Q(0,t),\\
&\label{equn_originPQ}F_1(0,t)=F_2(0,t),
\end{align}
for $t\geq0$. Direct computation yields
\begin{align}\label{equn_Pn-1Qb0}
P-(n-1)Q&=F_1-(n-1)F_2+(n-2)(n-1)(\xi(t)^{-1}-1)\\
&=-(n-1)(n-2)[b^{-2}(1-(b_s')^2)+1]=-(n-1)(n-2)[b^{-2}+1-\frac{H_g^2}{(n-1)^2}]\nonumber
\end{align}
with $H_g=H_g(r,t)$ the mean curvature of $\partial B_r(0)$ with respect to the metric $g(t)$, and hence,
\begin{align}\label{equn_Pn-1Qs}
\frac{\partial}{\partial s}(F_1-(n-1)F_2)=\frac{\partial}{\partial s}(P-(n-1)Q)=-2(P-Q)H_g=-2(F_1-F_2)H_g,
\end{align}
which plays an important role in later discussion on the signs of $F_1$ and $F_2$. In particular,
\begin{align}
&\label{equn_Pn-1Qb}F_1-(n-1)F_2=-(n-1)(n-2)[b^{-2}-\frac{\eta(t)^2}{(n-1)^2}+\xi(t)^{-1}],\\
&\label{equn_nPn-1Qb}\frac{\partial}{\partial n_g}(F_1-(n-1)F_2)=-2(F_1-F_2)\,\eta(t),
\end{align}
on $\partial M\times[0,T)$, where $n_g$ is the outer unit normal vector field. Taking derivative of $H_g(r,t)=(n-1)b^{-1}b_s'$ with respect to $t$ at fixed $r$, and substituting $(\ref{equn_NRFb})$ and the formula
\begin{align}\label{equn_bst}
(b_s')_t'=(b_s')_{ss}''+(n-3)b^{-1}b_s'(b_s')_s'+(n-2)b^{-2}[1-(b_s')^2]b_s'
\end{align}
 into the result, we obtain
\begin{align}
\partial_tH_g=H_g(F_1+(n-1)(1-\xi(t)^{-1}))-(n-1)\partial_sF_2,
\end{align}
and hence, we have
\begin{align}\label{equn_DtHb}
\eta'(t)=\eta(t) \,(F_1+(n-1)(1-\xi(t)^{-1}))\,-\,(n-1)\frac{\partial}{\partial n_g}F_2
\end{align}
on $\partial M\times[0,T)$.

We now compute how $F_1$ and $F_2$ evolve under the flow. Similar to $(\ref{equn_bst})$, we obtain
\begin{align*}
(b_{ss}'')_t'=&\,(b_{ss}'')_{ss}''+(n-3)b^{-1}b_s'(b_{ss}'')_s'-2b^{-1}(b_{ss}'')^2-(4n-9)b^{-2}(b_s')^2b_{ss}''\\
&+(n-2)b^{-2}b_{ss}''+(n-1)b_{ss}''-2(n-2)b^{-3}(b_s')^2(1-(b_s')^2).
\end{align*}
Recall that for a rotationally symmetric function $f$,
\begin{align*}
\Delta_{g(t)}f=f_{ss}''+(n-1)b^{-1}b_s'f_s'.
\end{align*}
Therefore, for $(r,t)\in(0,r_0]\times[0,T)$ and dimension $n\geq3$, we have that
\begin{align}
&\label{equn_envP}\partial_tF_1=\Delta_gF_1+2F_1(F_1+(n-1)(1-2\xi(t)^{-1}))-2(n-1)b^{-2}(F_1-F_2)-\frac{2(n-1)}{n-2}(F_1-F_2)^2,\\
&\label{equn_envQ}\partial_tF_2=\Delta_gF_2+2F_2(F_2+(n-1)(1-2\xi(t)^{-1}))+2b^{-2}(F_1-F_2)+\frac{2}{n-2}(F_1-F_2)^2,
\end{align}
where $g=g(t)$.

For $n=2$, we have
\begin{align}\label{equn_PQn2}
P=Q=F_1+(1-\xi(t)^{-1})=F_2+(1-\xi(t)^{-1})=-b^{-1}b_{ss}''+1=K_g+1=\frac{1}{2}(R_g+2),
\end{align}
where $K_g$ is the Gaussian curvature and $R_g$ is the scalar curvature, and hence for $n=2$,
\begin{align}\label{equn_n2Pt}
\partial_tF_1=\Delta_gF_1+2F_1(F_1+1-2\xi(t)^{-1}).
\end{align}
Let $k_g=b(r,t)^{-1}\partial_sb(r,t)$ be the geodesic curvature of $\partial B_r(0)$ with respect to the metric $g(t)$, which is the "mean curvature" of $\partial M$ for $n=2$. Then we obtain
\begin{align}
\partial_tk_g=k_g(F_1+1-\xi(t)^{-1})-\partial_sF_1,
\end{align}
and hence, we have
\begin{align}\label{equn_Dtkb}
\eta'(t)=\eta \,(F_1+1-\xi(t)^{-1})\,-\,\frac{\partial}{\partial n_g}F_1
\end{align}
on $\partial M\times[0,T)$. The initial-boundary problem $(\ref{equn_NRF1})-(\ref{equn_Ninc2m})$ reduces to the following form
\begin{align}
&\label{equn_NRF1d2}\frac{\partial}{\partial t}g=-2(K_g+1)g,\,\,\text{in}\,\,\overline{M}\times[0,\infty),\\
&\label{equn_Nbdc1d2}k_g=\eta,\,\,\,\,\,\,\,\,\text{on}\,\,\partial M\times[0,\infty),\\
&\label{equn_Ninc2d2}g\big|_{t=0}=m\,g_{-1},\,\,\,\,\text{on}\,\,\overline{M}.
\end{align}
By $(\ref{equn_aintt0})$, $(\ref{equn_bintt0})$ and $(\ref{equn_PQn2})$, we obtain
\begin{align}
&\label{equn_aPn2}a(r,t)=a(r,0)e^{-\int_0^tP(r,\tau)d\tau}=a(r,0)e^{-\int_0^tF_1(r,\tau)+(1-\xi(\tau)^{-1})d\tau},\\
&\label{equn_bPn2}b(r,t)=b(r,0)e^{-\int_0^tP(r,\tau)d\tau}=b(r,0)e^{-\int_0^tF_1(r,\tau)+(1-\xi(\tau)^{-1})d\tau},
\end{align}
for $(r,t)\in[0,r_0]\times[0,T)$. Let $u(r,t)=-\int_0^tP(r,\tau)d\tau$, for $(r,t)\in [0,r_0]\times[0,T)$. Therefore,
\begin{align*}
g(t)=e^{2u}g(0),
\end{align*}
for $t\in[0,T)$. Therefore, the initial-boundary value problem $(\ref{equn_NRF1d2})-(\ref{equn_Ninc2d2})$ is reduced to the following form (see also \cite{Li3})
\begin{align}
&\label{equn_ut}u_t=e^{-2u}(\Delta_{g(0)}u-K_{g(0)})-1,\,\,\text{in}\,\,\overline{M}\times[0,\infty),\\
&\label{equn_unb}\frac{\partial u}{\partial n_{g(0)}}+k_{g(0)}=\eta e^u,\,\,\text{on}\,\,\partial M\times[0,\infty),\\
&\label{equn_uint}u\big|_{t=0}=0,\,\,\,\text{in}\,\,\overline{M}.
\end{align}

\vskip0.2cm

\section{Preserving curvature upper bounds along the normalized Ricci flow}\label{section_3}
Let $\overline{M}$ be of dimension $n\geq3$. In this section we will show that the curvature functions defined in the preliminaries satisfy
\begin{align}\label{inequn_pin1}
(n-1)F_2\leq F_1\leq F_2\leq 0,
\end{align}
in $\overline{M}\times[0,+\infty)$ and the solution $g(t)$ exists for all time $t\geq0$. First, we show that $(\ref{inequn_pin1})$ holds for $t\geq0$ small.

\begin{lem}\label{lem_P_Q-0}
Let $n\geq3$. Let $g(t)$ be a solution to $(\ref{equn_NRF1})-(\ref{equn_Ninc2m})$ on $\overline{M}\times[0,T)$ for some $T>0$ under the compatibility conditions $(\ref{equn_comptmk})$ for some $k\geq2$. Let $\eta=\eta(t)\in C^k([0,\infty))$ satisfy the condition in Theorem \ref{thm_main}.  Then there exists $\epsilon_0\in(0,T)$ such that $(\ref{inequn_pin1})$ holds on $M\times[0,\epsilon_0]$.
\end{lem}
\begin{proof}
Subtracting $(\ref{equn_envQ})$ from $(\ref{equn_envP})$ yields
\begin{align}\label{equn_P-Qevol}
\partial_t(F_1-F_2)=\Delta_g(F_1-F_2)+2(F_1-F_2)[(n-1)(1-2\xi(t)^{-1})-nb^{-2}-\frac{2}{n-2}(F_1-(n-1)F_2)],
\end{align}
on $\overline{M}\times[0,T)$. Recall that $F_1(r,0)=F_2(r,0)=0$ for $0\leq r\leq r_0$. By continuity of $F_1,\,F_2$ and $b$, we have that there exists $\epsilon_1>0$ small such that
\begin{align*}
(n-1)(1-2\xi(t)^{-1})-nb^{-2}-\frac{2}{n-2}(F_1-(n-1)F_2)<(n-1)(1-2\xi(t)^{-1})
\end{align*}
for $0<r\leq r_0$ and $0\leq t\leq \epsilon_1$. Let $C=2(n-1)\max\{0,m-1\}$. Hence,
\begin{align}\label{equn_P-Qevolct}
&\partial_t[e^{-Ct}(F_1-F_2)]\\
=\,\,&\Delta_g[e^{-Ct}(F_1-F_2)]+2e^{-Ct}(F_1-F_2)[-\frac{C}{2}+(n-1)(1-2\xi(t)^{-1})-nb^{-2}-\frac{2}{n-2}(F_1-(n-1)F_2)],\nonumber
\end{align}
with
\begin{align*}
-\frac{C}{2}+(n-1)(1-2\xi(t)^{-1})-nb^{-2}-\frac{2}{n-2}(F_1-(n-1)F_2)<0
\end{align*}
on $\overline{M}\times[0,\epsilon_2]$ for some $\epsilon_2\in(0,\epsilon_1]$. Recall that by $(\ref{equn_originPQ})$, $F_1(0,t)-F_2(0,t)=0$ for $t\geq0$, and hence, by the maximum principle for $(\ref{equn_P-Qevolct})$ and continuity of $F_1-F_2$, the positive maximum and negative minimum of the function $e^{-Ct}(F_1-F_2)$, and hence of $F_1-F_2$, on $\overline{M}\times[0,t_1]$ for any $t_1\in(0,\epsilon_2)$ can only be achieved on $\partial M\times(0,t_1]$. Therefore, once it is proved that $F_1-F_2\leq0$ on $\partial M\times[0,T_1]$ for some $T_1\in(0,\epsilon_2)$, it follows that $F_1(r,t)-F_2(r,t)<0$ for $(r,t)\in(0,r_0)\times(0,T_1]$.

Now we define
\begin{align*}
\bar{b}(t)=b(r_0,0)e^{(n-1)\int_0^t(\xi(\tau)^{-1}-1)d\tau}=b(r_0,0)(m^{-1}\xi(t))^{\frac{1}{2}},
\end{align*}
for $t\geq0$. For $t>0$ small, by $(\ref{equn_bintt0})$ and the Taylor expansion, we have
\begin{align}
b(r_0,t)=\bar{b}(t)\,e^{-\int_0^tF_2(r_0,\tau)d\tau}=\bar{b}(t)[1\,-\,(1+o(1))\,\int_0^tF_2(r_0,\tau)d\tau]
\end{align}
where $o(1)\to0$ as $t\to0$. Thus, by $(\ref{equn_Pn-1Qb})$, $(\ref{equn_bintt0})$ and $(\ref{equn_comptmk})$,
\begin{align}\label{equn_Pn-1Qb1}
&F_1(r_0,t)-(n-1)F_2(r_0,t)\\
=&-(n-1)(n-2)[\xi(t)^{-1}+\bar{b}(t)^{-2}e^{2\int_0^tF_2(r_0,\tau)d\tau}]+\frac{n-2}{n-1}(H_{\bar{g}(t)}\big|_{\partial M}+\rho(t))^2\nonumber\\
=&-(n-1)(n-2)[\xi(t)^{-1}+\bar{b}(t)^{-2}(1+2(1+o(1))\int_0^tF_2(r_0,\tau)d\tau)]\nonumber\\
&+\frac{n-2}{n-1}((H_{\bar{g}(t)}\big|_{\partial M})^2+2\rho\,H_{\bar{g}(t)}\big|_{\partial M} (1+o(1)) )\nonumber\\
=&-2(n-1)(n-2)\bar{b}(t)^{-2}\int_0^tF_2(r_0,\tau)d\tau\,(1+o(1))\,+\,\frac{2(n-2)}{n-1}\rho\,H_{\bar{g}(t)}\big|_{\partial M}\,(1+o(1)),\nonumber
\end{align}
where $o(1)\to0$ as $t\to0$. Therefore, by $(\ref{equn_nPn-1Qb})$, $(\ref{equn_DtHb})$ and $(\ref{equn_Pn-1Qb1})$, we obtain
\begin{align}\label{equn_nP-Qb}
&\frac{\partial}{\partial n_g}(F_1-F_2)\\
=\,&(F_2-F_1)\eta-\frac{\eta}{n-1}(F_1-(n-1)F_2-(n-1)(n-2)(1-\xi(t)^{-1}))-\frac{n-2}{n-1}\eta'(t)\nonumber\\
=\,&(F_2-F_1)H_{\bar{g}(t)}\big|_{\partial M}\,(1+o(1))-\frac{H_{\bar{g}(t)}\big|_{\partial M}\,(1+o(1))}{n-1}(F_1-(n-1)F_2)\nonumber\\
&+(n-2)\rho(1-\xi(t)^{-1})-\frac{n-2}{n-1}\rho'(t)\nonumber\\
=\,&(F_2-F_1)H_{\bar{g}(t)}\big|_{\partial M}\,(1+o(1))+2(n-2)H_{\bar{g}(t)}\big|_{\partial M}\,\bar{b}(t)^{-2}\int_0^tF_2(r_0,\tau)d\tau\,(1+o(1))\nonumber\\
&-\frac{2(n-2)(H_{\bar{g}(t)}\big|_{\partial M})^2}{(n-1)^2}\rho(t)\,(1+o(1))+(n-2)\rho(t)\,(1-\xi(t)^{-1})-\frac{n-2}{n-1}\rho'(t)\nonumber\\
=\,&(F_2-F_1)H_{\bar{g}(t)}\big|_{\partial M}\,(1+o(1))+2(n-2)H_{\bar{g}(t)}\big|_{\partial M}\,\bar{b}(t)^{-2}\int_0^tF_2(r_0,\tau)d\tau\,(1+o(1))\nonumber\\
&+(n-2)\rho(t)[1-\xi(t)^{-1}-\frac{2(H_{\bar{g}(t)}\big|_{\partial M})^2}{(n-1)^2}](1+o(1))-\frac{n-2}{n-1}\rho'(t)\nonumber\\
\equiv&I_1+I_2+I_3+I_4,\nonumber
\end{align}
on $\partial M$ for $t>0$ small, where $o(1)\to0$ as $t\to0$. Therefore, by the conditions on $\rho(t)$ in Theorem \ref{thm_main}, there exist $\epsilon_3\in(0,\epsilon_2]$ and $\sigma_1>0$ such that
\begin{align}\label{inequn_i3i4}
I_3+I_4<-\sigma_1 \rho'(t)<0
\end{align}
for $t\in(0,\epsilon_3]$.

{\bf Claim:} there exists $\epsilon_4\in(0,\epsilon_3]$ such that $F_1-F_2\leq0$ on $\partial M\times[0,\epsilon_4]$.

Assume the contrary, i.e., for each $\delta>0$ small, there exists $t\in(0,\delta)$ such that
\begin{align}
F_1(r_0,t)-F_2(r_0,t)>0.
\end{align}
Thus, for each $\delta>0$, by the maximum principle for $(\ref{equn_P-Qevolct})$, 
there exists $t_1\in(0,\delta)$, such that
\begin{align}
&\label{inequn_P-Qbm}(F_1-F_2)(r_0,t_1)=\sup_{0\leq t\leq t_1}(F_1-F_2)(r_0,t)>0,\\
&\label{inequn_nP-Q0}\frac{\partial}{\partial n_g}(F_1-F_2)(r_0,t_1)\geq0,
\end{align}
and hence, the first equality in $(\ref{equn_nP-Qb})$ gives
 \begin{align*}
 F_2(r_0,t_1)&\geq \frac{\eta'(t_1)}{\eta(t_1)}+(n-1)(\xi(t)^{-1}-1)\\
 &=\frac{1}{\eta(t_1)}[\frac{d}{dt}(H_{\bar{g}}\big|_{\partial M}+\rho)+(n-1)(\xi(t_1)^{-1}-1)(H_{\bar{g}}\big|_{\partial M}+\rho)]\\
 &=\frac{1}{\eta(t_1)}[\rho'(t_1)+(n-1)\rho(t_1)(\xi(t_1)^{-1}-1)]>\frac{\sigma\,\rho'(t_1)}{(n-1+\sigma)\eta(t_1)},
 \end{align*}
 for $\delta<\epsilon$, where the last inequality follows from the condition on $\rho(t)$ in Theorem \ref{thm_main}. Also, it is clear that $I_i<0$ in $(\ref{equn_nP-Qb})$ at $(r_0,t_1)$ for $i=1,3,4$. Therefore, $I_2+I_3+I_4>0$ at $(r_0,t_1)$, and hence by $(\ref{inequn_i3i4})$,
\begin{align}\label{inequn_P-QQb}
\int_0^{t_1}F_2(r_0,\tau)d\tau>\frac{\bar{b}(t_1)^2}{2(n-2)H_{\bar{g}(t_1)}\big|_{\partial M}}\sigma_1 \rho'(t_1)\,(1+o(1))>0.
\end{align}
Thus, by continuity of $F_2$, there exists a smallest $t_2\in(0,t_1]$, such that
\begin{align}\label{inequn_Qb02}
F_2(r_0,t_2)=\sup_{0\leq t\leq t_1}F_2(r_0,t)> \frac{\bar{b}(t_1)^2}{2(n-2)t_1\,H_{\bar{g}(t_1)}\big|_{\partial M}}\sigma_1 \rho'(t_1)\,(1+o(1))>0.
\end{align}

 Combining $(\ref{inequn_P-Qbm})$,  $(\ref{equn_Pn-1Qb1})$ and $(\ref{inequn_Qb02})$, we have
\begin{align}\label{inequn_P-Qb01}
(F_1-F_2)(r_0,t_1)\geq&(F_1-F_2)(r_0,t_2)\\
=&(n-2)F_2(r_0,t_2)-2(n-1)(n-2)\bar{b}(t_2)^{-2}\int_0^{t_2}F_2(r_0,\tau)d\tau\,(1+o(1))\nonumber\\
&+\,\frac{2(n-2)}{n-1}\rho(t_2)\,H_{\bar{g}(t_2)}\big|_{\partial M}\,(1+o(1))\nonumber\\
=&(n-2)F_2(r_0,t_2)\,(1+o(1))+\frac{2(n-2)}{n-1}\rho(t_2)\,H_{\bar{g}(t_2)}\big|_{\partial M}\,(1+o(1))\nonumber\\
\geq&\frac{\bar{b}(t_1)^2}{2t_1\,H_{\bar{g}(t_1)}\big|_{\partial M}}\sigma_1 \rho'(t_1)\,(1+o(1))\nonumber
\end{align}
with $o(1)\to0$ as $\delta\to0$.

 Now at $t=t_1$, by $(\ref{inequn_P-Qb01})$ and $(\ref{inequn_Qb02})$, the terms $I_1$ and $I_2$ in $(\ref{equn_nP-Qb})$ satisfy
\begin{align*}
I_1&=(F_2-F_1)H_{\bar{g}(t_1)}\big|_{\partial M}\,(1+o(1))\\
&\leq -(n-2)F_2(r_0,t_2)\,(1+o(1)),\\
I_2&=2(n-2)H_{\bar{g}(t_1)}\big|_{\partial M}\,\bar{b}(t_1)^{-2}\int_0^{t_1}F_2(r_0,\tau)d\tau\,(1+o(1))\\
&\leq 2(n-2)t_1H_{\bar{g}(t_1)}\big|_{\partial M}\,\bar{b}(t_1)^{-2}F_2(r_0,t_2)\,(1+o(1)),
\end{align*}
with $o(1)\to0$ and $t_1\to0$ as $\delta\to0$, and hence, $I_1+I_2<0$ at $(r_0,t_1)$. Therefore, by $(\ref{inequn_Qb02})$, $(\ref{equn_nP-Qb})$ and the fact $I_i<0$ for $i=3,4$, we have
\begin{align}
\frac{\partial}{\partial n_g}(F_1-F_2)(r_0,t_1)<0,
\end{align}
for $\delta>0$ sufficiently small, contradicting with the assumption that $(F_1-F_2)$ achieves its maximum on $\overline{M}\times[0,t_1]$ at the point $(r_0,t_1)$. This proves the {\bf Claim}. Therefore,
\begin{align*}
F_1-F_2\leq 0
\end{align*}
for $\overline{M}\times[0,\epsilon_4]$, and by the maximum principle for $(\ref{equn_P-Qevolct})$,
\begin{align}\label{inequn_P-Q-0}
F_1(r,t)-F_2(r,t)< 0
\end{align}
for $(r,t)\in(0,r_0)\times(0,\epsilon_4]$.

We now proceed to prove that $F_2(r,t)\leq 0$ on $\overline{M}\times [0,\epsilon_5]$ for some $\epsilon_5\in(0,\epsilon_4]$. First, we verify it holds on the boundary. By $(\ref{equn_Pn-1Qb1})$, we obtain
\begin{align}\label{equn_QbODE0}
&F_2(r_0,t)\,-\,2(n-1)\bar{b}(t)^{-2}\int_0^tF_2(r_0,\tau)d\tau\,(1+o(1))\\
=\,&\frac{1}{n-2}(F_1(r_0,t)-F_2(r_0,t))\,-\,\frac{2}{n-1}\rho(t)\,H_{\bar{g}(t)}\big|_{\partial M}\,(1+o(1)),\nonumber
\end{align}
for $t\in(0,\epsilon_1]$, where $o(1)\to0$ as $t\to0$. Now we denote the right-hand side of $(\ref{equn_QbODE0})$ as $f(t)$, and hence $f(t)<0$ for $t\in(0,\epsilon_4]$. We denote $J(t)=\int_0^tF_2(r_0,\tau)d\tau$ and get the following first order differential equation for $\xi(t)$ for $t\in[0,\epsilon_4]$:
\begin{align}\label{equn_xibODE}
J'(t)-2(n-1)(1+o(1))\,\bar{b}(t)^{-2} J(t)=f(t)<0,
\end{align}
and hence by solving $(\ref{equn_xibODE})$, we obtain
\begin{align}
J(t)=e^{2(n-1)\int_0^t\bar{b}(\tau)^{-2}(1+o(1))d\tau}[\int_0^te^{-2(n-1)\int_0^{\tau}\bar{b}(\tilde{t})^{-2}(1+o(1))d\tilde{t}}f(\tau)d\tau]<0,
\end{align}
and hence by $(\ref{equn_xibODE})$,
\begin{align}\label{inequn_Qb03}
F_2(r_0,t)=J'(t)=2(n-1)(1+o(1))\,\bar{b}(t)^{-2} J(t)+f(t)<0,
\end{align}
for $t\in(0,\epsilon_4]$.

Now we show $F_2<0$ on $M\times(0,\epsilon_5]$ for some $\epsilon_5\in(0,\epsilon_4]$. Recall that $F_1(r,0)=F_2(r,0)=0$ for $0\leq r\leq r_0$. Since $(F_1-F_2)\leq0$ on $\overline{M}\times[0,\epsilon_4]$, by continuity of $F_1$ and $F_2$ and the equation $(\ref{equn_envQ})$, we have that there exists $\epsilon_5\in(0,\epsilon_4]$ such that
\begin{align}\label{inequn_Qt}
\partial_tF_2\leq\Delta_gF_2+2F_2(F_2+(n-1)(1-2\xi(t)^{-1}))
\end{align}
on $\overline{M}\times[0,\epsilon_5]$. By the maximum principle for $(\ref{inequn_Qt})$ and $(\ref{inequn_Qb03})$, we have
\begin{align}\label{inequn_Q-0}
F_2<0,\,\,\,\text{on}\,\,\overline{M}\times(0,\epsilon_5].
\end{align}
Thus, by $(\ref{equn_Pn-1Qb1})$, we obtain
\begin{align*}
(F_1-(n-1)F_2)(r_0,t)>0
\end{align*}
for $t\in(0,\epsilon_5)$. By $(\ref{inequn_Q-0})$ and $(\ref{equn_originPQ})$, we have
\begin{align}
(F_1-(n-1)F_2)(0,t)=-(n-2)F_2(0,t)>0
\end{align}
for $t\in(0,\epsilon_5]$; on the other hand, by $(\ref{inequn_P-Q-0})$, $(\ref{equn_Pn-1Qs})$ and Lemma \ref{lem_Hp}, we get
\begin{align}
\frac{\partial}{\partial s}(F_1-(n-1)F_2)=-2H_g(F_1-F_2)>0
\end{align}
for $(r,t)\in(0,r_0)\times(0,\epsilon_4]$. Therefore, we have
\begin{align}
F_1>(n-1)F_2
\end{align}
for $(r,t)\in[0,r_0]\times(0,\epsilon_5]$. Now take $\epsilon_0=\epsilon_5>0$.

In summary, we finally obtain that
\begin{align}
(n-1)F_2<F_1\leq F_2<0\,\,\,\text{on}\,\,\overline{M}\times(0,\epsilon_0],
\end{align}
and moreover,
\begin{align*}
F_1< F_2
\end{align*}
for $(r,t)\in (0,r_0)\times (0,\epsilon_0]$, and hence by the first equality in $(\ref{equn_nP-Qb})$, we have
\begin{align*}
F_1< F_2
\end{align*}
for $(r,t)\in (0,r_0]\times (0,\epsilon_0]$. This completes the proof of the lemma.

\end{proof}

Next we show that $(\ref{inequn_pin1})$ holds for all time.

\begin{thm}\label{thm_PQn-1Q2}
Let $n\geq3$. Let $g(t)$ be a solution to $(\ref{equn_NRF1})-(\ref{equn_Ninc2m})$ on $\overline{M}\times[0,T)$ for some $T>0$ under the compatibility conditions $(\ref{equn_comptmk})$ for some $k\geq2$. Assume that $\eta=\eta(t)\in C^k([0,\infty))$ satisfies the condition in Theorem \ref{thm_main}. Then $(\ref{inequn_pin1})$ holds on $M\times[0,T)$. Indeed, we have
\begin{align}\label{inequn_Q_P_Q1}
(n-1)F_2<F_1\leq F_2<0
\end{align}
on $\overline{M}\times(0,T)$, and moreover,
\begin{align}\label{inequn_P-Q-2}
F_1(r,t)<F_2(r,t)
\end{align}
for $(r,t)\in(0,r_0]\times(0,T)$. 

\end{thm}
\begin{proof}
By Lemma \ref{lem_P_Q-0}, the inequality $(\ref{inequn_pin1})$ holds on $\overline{M}\times[0,\epsilon_0]$ for some $\epsilon_0\in(0,T)$. Now we define the function
\begin{align*}
J(t)\triangleq \sup_{\overline{M}}\max\{F_2(\cdot,t),\,(F_1-F_2)(\cdot,t),\,((n-1)F_2-F_1)(\cdot,t)\}
\end{align*}
for $T> t\geq0$. Then $J(t)$ is continuous for $t\in [0,T)$. 

Assume, to the contrary, that there exists $t_0\geq \epsilon_0$, such that $(\ref{inequn_pin1})$ holds on $\overline{M}\times[0,t_0]$, while for each $\delta>0$, there exists $\tilde{t}\in(t_0,t_0+\delta)$ such that
\begin{align}\label{inequn_xi1}
J(\tilde{t})>0.
\end{align}
Since $F_1(0,t)=F_2(0,t)$ for $t\geq0$, we obtain
\begin{align}
J(t)=0,
\end{align}
for $t\in[0,t_0]$. Since $F_1-F_2\leq 0$ on $\overline{M}\times[0,t_0]$, applying the strong maximum principle for $(\ref{equn_P-Qevol})$ or $(\ref{equn_P-Qevolct})$ (see, for instance, Theorem 3 on Page 38 and Theorem 5 on Page 39 in \cite{Friedman}), we obtain
\begin{align}\label{inequn_P-Q-large}
(F_1-F_2)(r,t)<0
\end{align}
for $(r,t)\in(0,r_0)\times(0,t_0]$. If $F_2(r_1,t_0)=0$ for some $r_1\in[0,r_0]$, then by $(\ref{inequn_pin1})$ on $\overline{M}\times[0,t_0]$, we have \begin{align*}
F_1(r_1,t_0)=F_2(r_1,t_0)=(F_1-F_2)(r_1,t_0)=0,
\end{align*}
and hence by $(\ref{inequn_P-Q-large})$, $r_1=r_0$ or $r_1=0$.  Therefore,
\begin{align}\label{inequn_Q-large}
F_2(r,t_0)<0
\end{align}
for $r\in(0,r_0)$. Now if $F_2(0,t_0)=0$, since
\begin{align*}
(F_1-F_2)(0,t_0)=0,
\end{align*}
then by continuity of $(F_1-F_2)$, since $F_1-F_2\leq0$ on $\overline{M}\times[0,t_0]$, we have that $F_2$ satisfies the inequality $(\ref{inequn_Qt})$ in a neighborhood $U$ of the point $(0,t_0)$ in $\overline{M}\times[0,t_0]$. Since $F_2\leq 0$ on $U$ with $F_2(0,t_0)=0$, by the strong maximum principle for $(\ref{inequn_Qt})$, we have $F_2=0$ in $U$, contradicting with $(\ref{inequn_Q-large})$. Therefore,
\begin{align}\label{inequn_Q-large1}
F_2(r,t_0)<0
\end{align}
for $r\in[0,r_0)$, and hence, by $(\ref{equn_Pn-1Qs})$, $(\ref{inequn_P-Q-large})$ and Lemma \ref{lem_Hp}, we have
\begin{align*}
(F_1-(n-1)F_2)(r,t_0)\geq-(n-2)F_2(0,t_0)>0,
\end{align*}
for $r\in[0,r_0]$. The same discussion yields
\begin{align}\label{inequn_P_P-Q_Q-1}
(n-1)F_2<F_1\leq F_2<0
\end{align}
for $(r,t)\,\in\,[0,r_0)\times(0,t_0]$. Therefore, again, by $(\ref{inequn_P-Q-large})$, $(\ref{equn_Pn-1Qs})$ and Lemma \ref{lem_Hp}, we get
\begin{align}\label{inequn_Pn-1Q2}
F_1-(n-1)F_2>0
\end{align}
for $(r,t)\,\in\,[0,r_0]\times(0,t_0]$.

If $F_2(r_0,t)=0$ for some $t\in(0,t_0]$, then by $(\ref{inequn_pin1})$ on $\overline{M}\times[0,t_0]$, we have
\begin{align*}
F_1(r_0,t)=(F_1-F_2)(r_0,t)=(F_1-(n-1)F_2)(r_0,t)=0,
\end{align*}
contradicting with $(\ref{inequn_Pn-1Q2})$. Therefore, we have
\begin{align}\label{inequn_Q-2}
F_2<0
\end{align}
on $\overline{M}\times(0,t_0]$.

If $(F_1-F_2)(r_0,t)=0$ for some $t\in(0,t_0]$, then by $(\ref{inequn_P_P-Q_Q-1})$, we obtain
\begin{align*}
\frac{\partial}{\partial n_g}(F_1-F_2)(r_0,t)\geq0,
\end{align*}
contradicting with the first equality of $(\ref{equn_nP-Qb})$, by $(\ref{inequn_Pn-1Q2})$ and $(\ref{inequn_rhodt2})$.

In summary, we have that
\begin{align}
(n-1)F_2<F_1\leq F_2<0
\end{align}
on $\overline{M}\times(0,t_0]$, and moreover,
\begin{align}
F_1(r,t)<F_2(r,t)
\end{align}
for $(r,t)\in(0,r_0]\times(0,t_0]$. Therefore, by continuity, for $\sigma\in(0,r_0)$ small, there exists $\delta>0$, such that
\begin{align}
(n-1)F_2(r,t)<F_1(r,t)< F_2(r,t)<0
\end{align}
for $(r,t)\in [\sigma,r_0]\times[t_0,t_0+\delta]$, and also,
\begin{align}\label{inequn_Pn-1QQ1}
(n-1)F_2(r,t)-F_1(r,t)<0,\,\,\, F_2(r,t)<0,
\end{align}
for $(r,t)\in [0,r_0]\times[t_0,t_0+\delta]$. Thus,
\begin{align*}
-\frac{C}{2}+(n-1)(1-2\xi(t)^{-1})-nb^{-2}-\frac{2}{n-2}(F_1-(n-1)F_2)<0
\end{align*}
for $(r,t)\in(0,r_0]\times[0,t_0+\delta]$ and $C=2(n-1)\max\{0,m-1\}$. Recall that $(F_1-F_2)(0,t)=0$ for $t\geq0$. By the strong maximum principle for $(\ref{equn_P-Qevolct})$ on $(0,\sigma)\times[t_0,t_0+\delta]$, we obtain
\begin{align}
(F_1-F_2)(r,t)<0
\end{align}
for $(r,t)\in(0,\sigma)\times[t_0,t_0+\delta]$. Therefore,
\begin{align}\label{equn_P-Q2}
(F_1-F_2)(r,t)<0
\end{align}
for $(r,t)\in(0,r_0]\times(0,t_0+\delta]$.

By $(\ref{inequn_Pn-1QQ1})$, $(\ref{equn_P-Q2})$ and the fact $(F_1-F_2)(0,t)=0$ for $t\geq0$, we have
\begin{align*}
J(t)=0
\end{align*}
for $t\in[t_0,t_0+\delta]$, contradicting with $(\ref{inequn_xi1})$. Therefore, we obtain $(\ref{inequn_Q_P_Q1})$ on $\overline{M}\times(0,T)$, and $(\ref{inequn_P-Q-2})$ for $(r,t)\in(0,r_0]\times(0,T)$. This completes the proof of the theorem.

\end{proof}

We now turn to the proof of the long-time existence of the solution to the initial-boundary value problem $(\ref{equn_NRF1})-(\ref{equn_Ninc2m})$.
\begin{thm}\label{thm_existlongtime}
Let $n\geq3$ and $k\geq2$. Assume that $\eta=\eta(t)\in C^k([0,\infty))$ satisfies the condition in Theorem \ref{thm_main}. Then there exists a unique solution to $(\ref{equn_NRF1})-(\ref{equn_Ninc2m})$ on $\overline{M}\times[0,\infty)$, under the compatibility conditions $(\ref{equn_comptmk})$.
\end{thm}
\begin{proof}
By combining the variable transformation $(\ref{equn_vartransf})$, Theorem 1.2 and Theorem 1.3 in \cite{Gi}, we conclude that there exists a unique solution $g(t)\in C^{2k+\alpha,k+\frac{\alpha}{2}}(\overline{M}\times[0,\epsilon_0])$ to $(\ref{equn_NRF1})-(\ref{equn_Ninc2m})$ for some $\epsilon_0>0$. Let $T$ be the maximal existence time of the solution $g(t)$.

Now assume $T<\infty$. For any $t_0\in(0,T)$, since $(\ref{inequn_pin1})$ holds on $\overline{M}\times[0,T)$, by the maximum principle for $(\ref{equn_P-Qevolct})$, we have that there exists $t_1\in(0,t_0]$ such that
\begin{align}\label{equn_P-Qbi}
(F_1-F_2)(r_0,t_1)=\inf_{\overline{M}\times[0,t_0]}(F_1-F_2)<0,
\end{align}
and hence
\begin{align}
\frac{\partial}{\partial n_g}(F_1-F_2)(r_0,t_1)\leq 0
\end{align}
where $n_g$ is the outer normal vector field of $\partial M$. Thus, by the first equality in $(\ref{equn_nP-Qb})$, we have
\begin{align}\label{inequn_P-Qb2}
0\geq &(F_2-F_1)(r_0,t_1)\eta(t_1)-\frac{\eta(t_1)}{n-1}(F_1-(n-1)F_2-(n-1)(n-2)(1-\xi(t_1)^{-1}))-\frac{n-2}{n-1}\eta'(t_1)\\
=&(F_2-F_1)(r_0,t_1)\eta(t_1)-\frac{\eta(t_1)}{n-1}(F_1-(n-1)F_2)+(n-2)\rho(t_1)(1-\xi(t_1)^{-1})-\frac{n-2}{n-1}\rho'(t_1).\nonumber
\end{align}
On the other hand, by $(\ref{equn_Pn-1Qs})$, $(\ref{inequn_P-Q-2})$ and Lemma \ref{lem_Hp}, we obtain that $(F_1-(n-1)F_2)(r,t)>0$ is strictly increasing in $r$ for each $t>0$. The first equality in $(\ref{equn_Pn-1Qb1})$ yields
\begin{align}
&F_1(r_0,t)-(n-1)F_2(r_0,t)\nonumber\\
=&-(n-1)(n-2)[\xi(t)^{-1}+\bar{b}(t)^{-2}e^{2\int_0^tF_2(r_0,\tau)d\tau}]+\frac{n-2}{n-1}(H_{\bar{g}(t)}\big|_{\partial M}+\rho(t))^2\nonumber\\
\leq &\label{inequn_Pn-1Qb2}-(n-1)(n-2)\xi(t)^{-1}+\frac{n-2}{n-1}(H_{\bar{g}(t)}\big|_{\partial M}+\rho(t))^2
\end{align}
for $t\in(0,T)$, and hence,
\begin{align}\label{inequn_Pn-1Q3}
F_1(r,t)-(n-1)F_2(r,t)\leq -(n-1)(n-2)\xi(t)^{-1}+\frac{n-2}{n-1}(H_{\bar{g}(t)}\big|_{\partial M}+\rho(t))^2
\end{align}
for $(r,t)\in[0,r_0]\times(0,T)$. By $(\ref{inequn_P-Qb2})$ and $(\ref{inequn_Pn-1Qb2})$, we get
\begin{align}
&F_1(r_0,t_1)-F_2(r_0,t_1)\\
\geq &(n-2)\xi(t_1)^{-1}-\frac{n-2}{(n-1)^2}(H_{\bar{g}(t_1)}\big|_{\partial M}+\rho(t_1))^2+\frac{(n-2)\rho(t_1)}{\eta(t_1)}(1-\xi(t_1)^{-1})-\frac{n-2}{(n-1)\eta(t_1)}\rho'(t_1),\nonumber
\end{align}
and hence by $(\ref{equn_P-Qbi})$,
\begin{align}\label{inequn_P-Q3}
&F_1(r,t)-F_2(r,t)\\
\geq\,&(n-2)\xi(t_1)^{-1}-\frac{n-2}{(n-1)^2}(H_{\bar{g}(t_1)}\big|_{\partial M}+\rho(t_1))^2+\frac{(n-2)\rho(t_1)}{\eta(t_1)}(1-\xi(t_1)^{-1})-\frac{n-2}{(n-1)\eta(t_1)}\rho'(t_1),\nonumber
\end{align}
for $(r,t)\in [0,r_0]\times[0,t_0]$. Combining $(\ref{inequn_Pn-1Q3})$ with $(\ref{inequn_P-Q3})$, we have
that
\begin{align}
-C_1\leq (n-1)F_2<F_1\leq F_2<0
\end{align}
for $(r,t)\in[0,r_0]\times(0,t_0]$, where $C_1>0$ is uniformly bounded when $t_0\in(0,T)$. Since $g(t)$ is rotationally symmetric for $t>0$, the second fundamental form $II_g$ of $\partial M$ satisfies
\begin{align}
II_g=\frac{\eta(t)b(r_0,t)^2}{n-1}g_{\mathbb{S}^{n-1}}.
\end{align}
In summary, we have that there exists $C=C(T)>0$ such that
\begin{align}
&\sup_{M\times[0,T)}|Rm_{g(t)}|_{g(t)}+\sup_{\partial M\times[0,T)}|II_{g(t)}|_{g(t)}\leq C.
\end{align}
By $(\ref{equn_metricpolar}),\,(\ref{equn_aintt0})$ and $(\ref{equn_bintt0})$, we obtain that $a(r,t)(m\xi(t)^{-1})^{\frac{1}{2}}$ and $b(r,t)(m\xi(t)^{-1})^{\frac{1}{2}}$ are increasing in $t$, and there exists a constant $C_2=C_2(T)>0$ such that
\begin{align}
b(r_0,0)\leq b(r_0,0)e^{-\int_0^tF_2(r_0,\tau)d\tau}= b(r_0,t)e^{(n-1)\int_0^t(1-\xi(\tau)^{-1})d\tau}=b(r_0,t)(m\xi(t)^{-1})^{\frac{1}{2}}\leq C_2
\end{align}
for $t\in[0,T)$. Notice that $g^T(t)=b(r_0,t)^2g_{\mathbb{S}^{n-1}}$. By taking $\gamma(t)=g_{\mathbb{S}^{n-1}}$ for $t\in[0,T)$, one has that the boundary $\partial M$ is $\Lambda$-controlled in $(0,t_0]$ for any $t_0\in(0,T)$ (see Definition 3.1 in \cite{Gi1}). Hence, by Theorem 1.2 in \cite{Gi1}, for any $j=0,1,..,2k-2$, there exists a constant $C_3=C_3(T,k)>0$ such that
\begin{align}
&|\nabla^jRm_{g(t)}|_{g(t)}\leq C_3,\,\,\,\text{in}\,\,\overline{M}\times [0,T),\\
&|\nabla^{j+1}II_{g(t)}|_{g(t)}\leq C_3,\,\,\,\text{on}\,\,\partial M\times [0,T).
\end{align}
 By $(\ref{equn_aintt0})$ and $(\ref{equn_bintt0})$, the solution $g(t)$ can be extended to time $T$, and hence, by Theorem 1.2 and Theorem 1.3 in \cite{Gi}, the solution $g(t)$ exists on $[0,T+\delta]$ for some $\delta>0$, contradicting with the choice of $T$. Therefore, the solution $g(t)$ exists for all time $t>0$. This completes the proof of the theorem.

\end{proof}

\vskip0.2cm

\section{Locally uniform convergence of the solution to $(\ref{equn_NRF1})-(\ref{equn_Ninc2m})$ in $M$}\label{section_4}

The discussion in this section is a extension of that presented in Section 4 in \cite{Li4}. Let $\overline{M}$ be of dimension $n\geq3$. In this section, we will show that both $a(r,t)$ and $b(r,t)$ in the decomposition $(\ref{equn_metricpolar1})$ and the curvatures terms $F_1(r,t)$ and $F_2(r,t)$ are uniformly bounded on $[0,r_1]\times[0,\infty)$ for any $r_1\in(0,r_0)$; moreover, as $t\to\infty$, these functions converge locally uniformly in $r$ on $[0,r_0)$.

First, we show that $a(r,t)$, $b(r,t)$ and the distance function $s=\int_0^ra(x,t)dx$ are uniformly bounded on $[0,r_1]\times[0,\infty)$ for any $r_1\in(0,r_0)$.
\begin{lem}\label{lem_absBD}
Let $n\geq3$ and $k\geq2$. Under the assumption in Theorem \ref{thm_main}, for each $r_1\in(0,r_0)$, there exists a constant $C=C(m,r_1)>0$ such that
\begin{align}
&\label{inequn_LBa}a(r,0)\leq a(r,t)(m\xi(t)^{-1})^{\frac{1}{2}}\leq C,\\
&\label{inequn_LBb}b(r,0)\leq b(r,t)(m\xi(t)^{-1})^{\frac{1}{2}}\leq Cr,\\
&\label{inequn_LBs}s(r,0)\leq s(r,t)(m\xi(t)^{-1})^{\frac{1}{2}}\leq Cr,
\end{align}
for $(r,t)\in[0,r_1]\times[0,\infty)$.
\end{lem}
\begin{proof}
It follows from Theorem \ref{thm_PQn-1Q2} and Theorem \ref{thm_existlongtime} that $(\ref{inequn_Q_P_Q1})$ holds on $\overline{M}\times(0,\infty)$, and $(\ref{inequn_P-Q-2})$ holds for $(r,t)\in(0,r_0]\times(0,\infty)$. Thus, using the monotonicity of both $(m\xi(t)^{-1})^{\frac{1}{2}}a$ and $(m\xi(t)^{-1})^{\frac{1}{2}}b$ in $t$, and by $(\ref{equn_saformul})$, we have that the first inequality in each of the three chains $(\ref{inequn_LBa})-(\ref{inequn_LBs})$ holds.

Since $F_1(r,t)<0$ for $t>0$, by definition,
\begin{align*}
&b_{ss}''-\xi(t)^{-1}b\geq0,
\end{align*}
and hence,
\begin{align}
(e^{\xi(t)^{-\frac{1}{2}}s}b_s')_s'-(\xi(t)^{-\frac{1}{2}}e^{\xi(t)^{-\frac{1}{2}}s}b)_s'\geq0,
\end{align}
for $(r,t)\in[0,r_0]\times[0,\infty)$. Here $s=\int_0^ra(x,t)dx$, for $t\geq0$. Integrating this inequality with respect to $s$ and by the fact $b_s'(0,t)=1$ for any $t\geq0$, we obtain
\begin{align}\label{inequn_bsb1}
b_s'-\xi(t)^{-\frac{1}{2}}b-e^{-\xi(t)^{-\frac{1}{2}}s}\geq0
\end{align}
on $\overline{M}\times[0,\infty)$, and hence by the fact $b(0,t)=0$ for $t\geq0$,
\begin{align}\label{inequn_bsinhslbd}
b\geq \xi(t)^{\frac{1}{2}}\sinh(\xi(t)^{-\frac{1}{2}}s)
\end{align}
on $\overline{M}$ for any $t\geq0$. Thus, by $(\ref{inequn_bsb1})$,
\begin{align}
b_s'\geq \cosh(\xi(t)^{-\frac{1}{2}}s)
\end{align}
on $\overline{M}$ for any $t\geq0$.

On the other hand, by $(\ref{equn_aintt0})$, $(\ref{equn_bintt0})$ and $(\ref{inequn_Q_P_Q1})$,
\begin{align}\label{inequn_bn-1ab2}
(n-1)\ln(\frac{b(r,t)m^{\frac{1}{2}}}{b(r,0)\xi(t)^{\frac{1}{2}}})&=-(n-1)\int_0^tF_2(r,\tau)d\tau>-\int_0^tF_1(r,\tau)d\tau
= \ln(\frac{a(r,t)m^{\frac{1}{2}}}{a(r,0)\xi(t)^{\frac{1}{2}}})\\
&\geq -\int_0^tF_2(r,\tau)d\tau =\ln(\frac{b(r,t)m^{\frac{1}{2}}}{b(r,0)\xi(t)^{\frac{1}{2}}}) \nonumber
\end{align}
for $(r,t)\in(0,r_0]\times(0,\infty)$, and hence, by Lemma \ref{lem_Hp}, for any $0< r_1<r_2\leq r_0$,
\begin{align*}
\ln(\frac{a(r_2,t)m^{\frac{1}{2}}}{a(r_2,0)\xi(t)^{\frac{1}{2}}})\geq \ln(\frac{b(r_2,t)m^{\frac{1}{2}}}{b(r_2,0)\xi(t)^{\frac{1}{2}}})&\geq \ln(\frac{b(r_1,t)m^{\frac{1}{2}}}{b(r_2,0)\xi(t)^{\frac{1}{2}}})
=\ln(\frac{b(r_1,t)m^{\frac{1}{2}}}{b(r_1,0)\xi(t)^{\frac{1}{2}}})+\ln(\frac{b(r_1,0)}{b(r_2,0)})\\
&\geq \frac{1}{n-1}\ln(\frac{a(r_1,t)m^{\frac{1}{2}}}{a(r_1,0)\xi(t)^{\frac{1}{2}}})+\ln(\frac{b(r_1,0)}{b(r_2,0)})
\end{align*}
for $t>0$. Therefore,
\begin{align}\label{inequn_a2a1}
\frac{a(r_2,t)m^{\frac{1}{2}}}{a(r_2,0)\xi(t)^{\frac{1}{2}}}\geq\,\frac{b(r_1,0)}{b(r_2,0)}\,\big(\frac{a(r_1,t)m^{\frac{1}{2}}}{a(r_1,0)\xi(t)^{\frac{1}{2}}}\big)^{\frac{1}{n-1}}
\end{align}
for any $0\leq r_1<r_2\leq r_0$ and $t>0$.

Now, by $(\ref{inequn_bn-1ab2})$ and $(\ref{inequn_bsinhslbd})$,
\begin{align*}
\frac{\partial s}{\partial r}=a(r,t)\geq \frac{a(r,0)}{b(r,0)} b(r,t) \geq \frac{a(r,0)}{b(r,0)} \xi(t)^{\frac{1}{2}}\sinh(\xi(t)^{-\frac{1}{2}}s),
\end{align*}
and hence,
\begin{align}\label{inequn_sabr1}
\frac{ds}{\xi(t)^{\frac{1}{2}}\sinh(\xi(t)^{-\frac{1}{2}}s)}\geq \frac{a(r,0)}{b(r,0)} dr
\end{align}
for $t>0$. For any $0< r_1\leq r_2\leq r_0$, integrating $(\ref{inequn_sabr1})$ over $r\in[r_1,r_2]$ yields
\begin{align}
\ln(\frac{e^{\xi(t)^{-\frac{1}{2}}s(r_2,t)}-1}{e^{\xi(t)^{-\frac{1}{2}}s(r_2,t)}+1})-\ln(\frac{e^{\xi(t)^{-\frac{1}{2}}s(r_1,t)}-1}{e^{\xi(t)^{-\frac{1}{2}}s(r_1,t)}+1})\geq \int_{r_1}^{r_2}a(r,0)b^{-1}(r,0)dr
\end{align}
for $t>0$. Thus,
\begin{align}
1>\frac{e^{\xi(t)^{-\frac{1}{2}}s(r_2,t)}-1}{e^{\xi(t)^{-\frac{1}{2}}s(r_2,t)}+1}&\geq\,\frac{e^{\xi(t)^{-\frac{1}{2}}s(r_1,t)}-1}{e^{\xi(t)^{-\frac{1}{2}}s(r_1,t)}+1}\,e^{\int_{r_1}^{r_2}a(r,0)b^{-1}(r,0)dr}\\
&=\,\frac{e^{\xi(t)^{-\frac{1}{2}}s(r_1,t)}-1}{e^{\xi(t)^{-\frac{1}{2}}s(r_1,t)}+1}\,\frac{e^{r_2}-1}{e^{r_2}+1}\,\frac{e^{r_1}+1}{e^{r_1}-1}
\end{align}
for $t>0$. Now take $r_2=r_0$, and hence
\begin{align}
&\frac{e^{\xi(t)^{-\frac{1}{2}}s(r,t)}-1}{e^{\xi(t)^{-\frac{1}{2}}s(r,t)}+1}<\frac{e^{r_0}+1}{e^{r_0}-1}\,\frac{e^{r}-1}{e^{r}+1},\\
&\label{inequn_supbd2}s(r,t)<\xi(t)^{\frac{1}{2}}\ln(\frac{e^{r_0+r}-1}{e^{r_0}-e^r})<\infty,
\end{align}
for any $(r,t)\in(0,r_0)\times(0,\infty)$, which proves the second inequality in $(\ref{inequn_LBs})$. Thus,
\begin{align*}
a(0,t)=\lim_{r\to0_+}\frac{s(r,t)}{r}\leq\frac{e^{r_0}+1}{e^{r_0}-1}\xi(t)^{\frac{1}{2}},
\end{align*}
for any $t>0$.

Next, we show that for any $r_1\in(0,r_0)$, $a(r,t)$ and $b(r,t)$ are uniformly bounded on $[0,r_1]\times[0,\infty)$. Assume for contradiction that there exists a sequence of points $\{(\bar{r}_j,t_j)\}_{j=1}^\infty$ such that $a(\bar{r}_j,t_j)\to\infty$, $\epsilon\leq \bar{r}_j\leq r_1$ and $t_j\to\infty$, for some $\epsilon\in(0,r_1)$. Then for any $r\in[r_1,\frac{r_1+r_0}{2}]$, taking $r_1=\bar{r}_j$ and $r_2=r$ in $(\ref{inequn_a2a1})$, one gets
\begin{align}
a(r,t_j)\geq\,\frac{\xi(t_j)^{\frac{1}{2}}a(r,0)b(\bar{r}_j,0)}{m^{\frac{1}{2}}b(r,0)}\,\big(\frac{m^{\frac{1}{2}}a(\bar{r}_j,t_j)}{\xi(t_j)^{\frac{1}{2}}a(\bar{r}_j,0)}\big)^{\frac{1}{n-1}},
\end{align}
for $r\in[r_1,\frac{r_1+r_0}{2}]$, and hence,
\begin{align}
s(\frac{r_1+r_0}{2},t_j)-s(r_1,t_j)=\int_{r_1}^{\frac{r_1+r_0}{2}}a(r,t_j)dr\to\infty
\end{align}
as $t_j\to\infty$, contradicting to $(\ref{inequn_supbd2})$. Therefore, there exists $C=C(r_1,\epsilon)>0$ such that
\begin{align}
a(r,t)\leq C
\end{align}
for $(r,t)\in[\epsilon,r_1]\times[0,\infty)$, and hence, by $(\ref{inequn_bn-1ab2})$,
\begin{align}\label{inequn_bupbb1}
b(r,t)\leq \frac{C}{a(r,0)}b(r,0)=C\sinh(r)
\end{align}
for $(r,t)\in[\epsilon,r_1]\times[0,\infty)$. By $b_r'(r,t)>0$, we obtain $b(r,t)\leq C\sinh(r_1)$ for $(r,t)\in[0,r_1]\times[0,\infty)$.

We next show that $F_1-(n-1)F_2$ is uniformly bounded from above on $[0,r_1]\times[0,\infty)$ for any $r_1\in(0,r_0)$.

Assume, to the contrary, that there exists a sequence$\{\bar{r}_j\}_{j=1}^\infty$ on $[0,r_1]$ and $\{t_j\}_{j=1}^\infty\subseteq [0,\infty)$ with $t_j\to\infty$ as $j\to\infty$ such that
\begin{align*}
(F_1-(n-1)F_2)(\bar{r}_j,t_j)\to \infty
\end{align*}
as $j\to\infty$. Since $F_1-(n-1)F_2$ is increasing in $r$ on $[0,r_0]$ for any $t>0$, we obtain
\begin{align}\label{equn_Pn-1Qblup}
(F_1-(n-1)F_2)(r,t_j)=-(n-1)(n-2)[b(r,t_j)^{-2}(1-(b_s'(r,t_j))^2)+\xi(t_j)^{-1}]\to +\infty
\end{align}
uniformly in $r$ on $[r_1,\frac{r_1+r_0}{2}]$ as $j\to\infty$. Since $b(r,t)$ is uniformly bounded from below on $[r_1,\frac{r_1+r_0}{2}]\times[0,\infty)$ by the first inequality in $(\ref{inequn_LBb})$, we obtain
\begin{align*}
b_s'(r,t_j)\to \infty
\end{align*}
uniformly in $r$ on $[r_1,\frac{r_1+r_0}{2}]$, as $j\to\infty$. Therefore, by inequality $(\ref{inequn_LBs})$ on $[0,\frac{r_1+r_0}{2}]\times[0,\infty)$,
\begin{align*}
b(\frac{r_1+r_0}{2},t_j)-b(r_1,t_j)=\int_{s(r_1,t_j)}^{s(\frac{r_1+r_0}{2},t_j)}b_s'(r,t_j)ds\to\infty,
\end{align*}
as $j\to\infty$, contradicting with $(\ref{inequn_bupbb1})$. 
 Thus, for any $r_1\in(0,r_0)$, there exists $C_1=C_1(r_1)>0$ such that
 \begin{align}
 &0\leq (F_1-(n-1)F_2)(r,t)\leq C_1,\\
 &\label{inequn_b'supbd}b_s'(r,t)\leq C_1,
 \end{align}
for $(r,t)\in[0,r_1]\times[0,\infty)$, where for $(\ref{inequn_b'supbd})$ we have used the equality in $(\ref{equn_Pn-1Qblup})$ again. By the definition of $F_1$, since $F_1<0$ on $\overline{M}\times(0,\infty)$ and $b_s'(0,t)=1$ for $t\geq0$, we obtain that $b_s'$ is increasing in $s$ and
\begin{align}
1\leq b_s'(r,t)\leq C_1
\end{align}
for $(r,t)\in[0,r_1]\times[0,\infty)$. Therefore,
\begin{align*}
s\leq b(r,t)=\int_0^sb_s'(r,t)ds \leq C_1s\leq C_1C_2r
\end{align*}
with some constant $C_2=C_2(r_1)>0$, for $(r,t)\in[0,r_1]\times[0,\infty)$, where the last inequality is by $(\ref{inequn_supbd2})$. Thus, by $(\ref{inequn_bn-1ab2})$, we have
\begin{align*}
a(r,t)\leq (\frac{b(r,t)}{b(r,0)})^{n-1}(\frac{m}{\xi(t)})^{\frac{n-2}{2}} a(r,0)\leq C'
\end{align*}
with some constant $C'=C'(r_1)>0$, for $(r,t)\in [0,r_1]\times[0,\infty)$. This completes the proof of the lemma.

\end{proof}

Next, for each $r_1\in(0,r_0)$, we show that $F_1(r,t)$ and $F_2(r,t)$ are uniformly bounded on $[0,r_1]\times[0,\infty)$, and moreover, $F_1(r,t)$ and $F_2(r,t)$ converge to $0$ uniformly for $r\in[0,r_1]$ as $t\to\infty$.

\begin{thm}\label{thm_QPn-1Qunibd2}
Let $n\geq3$ and $k\geq2$. Assume that $\eta=\eta(t)\in C^k([0,\infty))$ satisfies the condition in Theorem \ref{thm_main}. Assume the compatibility conditions $(\ref{equn_comptmk})$ hold. Then for any $r_1\in(0,r_0)$, there exists a constant $C=C(r_1)>0$ such that
\begin{align}\label{inequn_Qlb2}
-C\leq (n-1)F_2<F_1\leq F_2<0
\end{align}
on $[0,r_1]\times(0,\infty)$. Moreover, for each $j\in \mathbb{N}$ and $r_1\in(0,r_0)$, there exists $C=C(n,j,r_1)>0$, such that
\begin{align*}
|\nabla^jRm_{g(t)}|_{g(t)}\leq C
\end{align*}
on $\overline{B_{r_1}}(0)\times[0,\infty)$.
\end{thm}
\begin{proof}

We first prove that for any $r_1\in(0,r_0)$, $F_1(r,t)$ and $F_2(r,t)$ are uniformly bounded from below on $[0,r_1]\times[0,\infty)$.

By $(\ref{equn_envP})$ and $(\ref{equn_envQ})$, we obtain
\begin{align}\label{equn_Pn-1Qt}
\partial_t(F_1+(n-1)F_2)=\Delta_g(F_1+(n-1)F_2)+2(1-n)(2\xi(t)^{-1}-1)(F_1+(n-1)F_2)+2(F_1^2+(n-1)F_2^2)
\end{align}
on $\overline{M}\times[0,\infty)$, where
\begin{align*}
\frac{2}{n}(F_1+(n-1)F_2)^2\leq 2(F_1^2+(n-1)F_2^2)\leq 2(F_1+(n-1)F_2)^2,
\end{align*}
by $(\ref{inequn_Q_P_Q1})$, and hence,
\begin{align}\label{equn_Pn-1Qt2}
\partial_t(F_1+(n-1)F_2)=\Delta_g(F_1+(n-1)F_2)+2((1-n)(2\xi(t)^{-1}-1)+f)(F_1+(n-1)F_2)
\end{align}
on $\overline{M}\times[0,\infty)$, with
\begin{align*}
f= \frac{(F_1^2+(n-1)F_2^2)}{(F_1+(n-1)F_2)}<0,
\end{align*}
continuous on $\overline{M}\times(0,\infty)$. By the maximum principle for $(\ref{equn_Pn-1Qt2})$, since $2\xi(t)^{-1}-1>-1$, as the argument in the first paragraph in the proof of Lemma \ref{lem_P_Q-0}, once $F_1+(n-1)F_2$ has a uniform lower bound on $\{r_1\}\times[0,\infty)$, then $F_1+(n-1)F_2$ is uniformly bounded from below on $[0,r_1]\times[0,\infty)$.

Assume, for contradiction, that there exists a sequence $t_j\nearrow+\infty$ as $j\to\infty$, such that $(F_1+(n-1)F_2)(r_1,t_j)\to -\infty$. Then by $(\ref{inequn_pin1})$,
\begin{align}
F_1(r_1,t_j)\to-\infty,\,\,\,F_2(r_1,t_j)\to-\infty
\end{align}
as $j\to\infty$. Since $F_1-(n-1)F_2$ is uniformly bounded on $[0,r_1]\times[0,\infty)$, we obtain
\begin{align}\label{inequn_P-Q2blup}
F_1(r_1,t_j)-F_2(r_1,t_j)\to-\infty
\end{align}
as $j\to\infty$. Let $\delta=\frac{1}{2}\min\{r_1, r_0-r_1\}$. On $[r_1-\delta, r_1+\delta]\times(0,\infty)$, we rewrite $(\ref{equn_P-Qevol})$ as
\begin{align*}
\partial_t(F_1-F_2)=&a^{-2}\partial_r^2(F_1-F_2)+[-a^{-3}\partial_ra\,+(n-1)a^{-1}b^{-1}b_s']\,\partial_r(F_1-F_2)\\
&+2[(n-1)(1-2\xi(t)^{-1})-nb^{-2}-\frac{2}{n-2}(F_1-(n-1)F_2)](F_1-F_2).
\end{align*}
Now we define the continuous functions
\begin{align*}
&f_1(r,t)=-a^{-3}\partial_ra\,+(n-1)a^{-1}b^{-1}b_s',\\
&f_2(r,t)=2[(n-1)(1-2\xi(t)^{-1})-nb^{-2}-\frac{2}{n-2}(F_1-(n-1)F_2)],
\end{align*}
on $[r_1-\delta, r_1+\delta]\times(0,\infty)$. Therefore,
\begin{align}\label{equn_P-Qt2}
\partial_t(F_1-F_2)=a^{-2}\partial_r^2(F_1-F_2)+f_1(r,t)\,\partial_r(F_1-F_2)+f_2(r,t)(F_1-F_2).
\end{align}

{\bf Claim:} there exists $C=C(r_1)>0$ such that
\begin{align}
|\partial_ra|\leq C
\end{align}
on $[r_1-\delta, r_1+\delta]\times(0,\infty)$.

Once the claim holds, since $F_1-(n-1)F_2$ and $b_s'$ are uniformly bounded on $[r_1-\delta, r_1+\delta]\times(0,\infty)$, by Lemma \ref{lem_absBD}, there exists $C_2>0$ such that
\begin{align*}
|f_1(r,t)|+|f_2(r,t)|\leq C_2
\end{align*}
on $[r_1-\delta, r_1+\delta]\times(0,\infty)$. Hence, it follows from the parabolic Harnack inequality for $(\ref{equn_P-Qt2})$ that there exists $C>0$ independent of $j\in \mathbb{N}$ such that
\begin{align}
(F_2-F_1)(r,t)\geq C(F_2-F_1)(r_1,t_j)>0
\end{align}
for $(r,t)\in U_{j,\frac{\delta}{2}}=\{(r,t):\,|r-r_1|\leq \frac{\delta}{2},\,\,t_j+\frac{\delta^2}{8}\leq t\leq t_j+\frac{\delta^2}{4}\}$ and any $t_j\geq \delta^2$. We choose a subsequence $\{t_{j_k}\}_{k=1}^\infty\subseteq \{t_j\}_{j=1}^\infty$ such that $t_{j_1}>\delta^2$ and $t_{j_{k+1}}>t_{j_k}+\delta^2$, for $k\geq1$. Since $F_2<0$ for $t>0$, we get
\begin{align}
-F_1(r,t)\geq C(F_2-F_1)(r_1,t_{j_k})>0
\end{align}
for $(r,t)\in U_{j_k,\frac{\delta}{2}}$ and any $k\in \mathbb{N}$, and hence by $(\ref{inequn_P-Q2blup})$,
\begin{align*}
a(r_1,t_{j_p})=a(r_1,0)(m^{-1}\xi(t_{j_p}))^{\frac{1}{2}}e^{\int_0^{t_{j_p}}-F_1(r_1,\tau)d\tau}&\geq a(r_1,0)(m^{-1}\xi(t_{j_p}))^{\frac{1}{2}}e^{C\sum_{k=1}^{p-1}\int_{t_{j_k}+\frac{\delta^2}{8}}^{t_{j_k}+\frac{\delta^2}{4}}(F_2-F_1)(r_1,t_{j_k})d\tau}\\
&=a(r_1,0)(m^{-1}\xi(t_{j_p}))^{\frac{1}{2}}e^{\frac{C\delta^2}{8}\sum_{k=1}^{p-1}(F_2-F_1)(r_1,t_{j_k})}\to\infty,
\end{align*}
as $p\to\infty$, contradicting with Lemma \ref{lem_absBD}, which proves the uniform boundedness of $F_1+(n-1)F_2$, and hence $F_1$ and $F_2$, on $[0,r_1]\times[0,\infty)$.

Now, it remains to prove the claim.

It follows by direct calculation that
\begin{align*}
&\partial_t(a_r')=\partial_r(\partial_ta)=-\partial_r((F_1+(n-1)(1-\xi(t)^{-1}))a)=-(F_1+(n-1)(1-\xi(t)^{-1}))\partial_ra-a\partial_rF_1,\\
&\partial_t(b_r')=\partial_r(\partial_tb)=-\partial_r((F_2+(n-1)(1-\xi(t)^{-1}))b)=-(F_2+(n-1)(1-\xi(t)^{-1}))\partial_rb-b\partial_rF_2,
\end{align*}
and hence,
\begin{align*}
\partial_t(a^{-1}a_r'-(n-1)b^{-1}b_r')=\partial_r((n-1)F_2-F_1)&=a\partial_s((n-1)F_2-F_1)\\
&=2(n-1)ab^{-1}b_s'\,(F_1-F_2)
\end{align*}
where we have used $(\ref{equn_NRFa})$, $(\ref{equn_NRFb})$, $(\ref{equn_Pn-1Qs})$ and the definition of $F_1$ and $F_2$. Integrating both sides of this equation with respect to $t$, we get
\begin{align}\label{equn_arbr}
\,\,\,&a(r,t)^{-1}a_r'(r,t)-(n-1)b(r,t)^{-1}b_r'(r,t)\\
=\,\,\,&a(r,0)^{-1}a_r'(r,0)-(n-1)b(r,0)^{-1}b_r'(r,0)\nonumber\\
+2&(n-1)\int_0^ta(r,\tau)b(r,\tau)^{-1}b_s'(r,\tau)F_1(r,\tau)d\tau-2(n-1)\int_0^ta(r,\tau)b(r,\tau)^{-1}b_s'(r,\tau)F_2(r,\tau)d\tau,\nonumber
\end{align}
on $[r_1-\delta, r_1+\delta]\times(0,\infty)$. On the other hand, $a,\,a^{-1},\,b,\,b^{-1},\,b_s'$, and hence $b_r'=ab_s'$, are uniformly bounded on $[r_1-\delta, r_1+\delta]\times(0,\infty)$, and moreover, we have $F_1<F_2<0$ on $[r_1-\delta, r_1+\delta]\times(0,\infty)$, and the integrals
\begin{align*}
&\int_0^t-F_1(r,\tau)d\tau=\ln(\frac{a(r,t)}{a(r,0)(m^{-1}\xi(t))^{\frac{1}{2}}}),\\
&\int_0^t-F_2(r,\tau)d\tau=\ln(\frac{b(r,t)}{b(r,0)(m^{-1}\xi(t))^{\frac{1}{2}}})
\end{align*}
are uniformly bounded on $[r_1-\delta, r_1+\delta]\times(0,\infty)$. Therefore, by $(\ref{equn_arbr})$, there exists a constant $C_3=C_3(r_1)>0$ such that
\begin{align}
|a_r'(r,t)|\leq C_3
\end{align}
on $[r_1-\delta, r_1+\delta]\times(0,\infty)$. This proves the {\bf Claim}. Therefore, $(\ref{inequn_Qlb2})$ holds. By Shi's interior estimates in \cite{Shi} on the solution to the Ricci flow, for each $j\geq0$ and $r_1\in(0,r_0)$, there exists $C=C(n,j,r_1)>0$ such that
\begin{align*}
|\nabla^jRm_{g(t)}|_{g(t)}\leq C
\end{align*}
on $\overline{B_{r_1}}(0)\times[0,\infty)$. This completes the proof of the theorem.

\end{proof}

\begin{thm}
Let $n\geq3$ and $k\geq2$. Assume that $\eta=\eta(t)\in C^k([0,\infty))$ satisfies the condition in Theorem \ref{thm_main}. Assume the compatibility conditions $(\ref{equn_comptmk})$ hold. Then $P$ and $Q$ converge to zero uniformly in $r$ on $[0,r_1]$ as $t\to\infty$, for any $r_1\in(0,r_0)$. Moreover, $g(t)$ converges locally uniformly in the interior of $\overline{M}$ to a rotationally symmetric locally hyperbolic metric $g_{\infty}$, as $t\to\infty$.
\end{thm}
\begin{proof}
By the proof of Theorem \ref{thm_QPn-1Qunibd2}, for each $r_1\in(0,r_0)$, $F_1, \,F_2$, $b_s'$, and hence the function $f$ in $(\ref{equn_Pn-1Qt2})$, are uniformly bounded on $[0,r_1]\times[0,\infty)$, and moreover, by the claim there, $|a_r(r,t)|$ is uniformly bounded on $[r_1-\delta,r_1+\delta]\times[0,\infty)$.

Now assume there exists a constant $\epsilon>0$ and an increasing sequence $\{t_j\}_{j=1}^\infty$ such that $t_j\to\infty$ and $(F_1+(n-1)F_2)(r_1,t_j)<-\epsilon$. We rewrite $(\ref{equn_Pn-1Qt2})$ as
\begin{align}\label{equn_F1n-1F2t3}
&\partial_t(F_1+(n-1)F_2)-a^{-2}\partial_r^2(F_1+(n-1)F_2)\\
=\,\,&f_3(r,t)\partial_r(F_1+(n-1)F_2)+f_4(r,t)(F_1+(n-1)F_2)\nonumber
\end{align}
on $[r_1-\delta,r_1+\delta]\times[0,\infty)$, where $f_3=-a^{-3}\partial_ra\,+(n-1)a^{-1}b^{-1}b_s'$ and $f_4=2((1-n)(2\xi(t)^{-1}-1)+f)$ are uniformly bounded on $[r_1-\delta,r_1+\delta]\times[0,\infty)$. Recall that $F_1+(n-1)F_2\leq0$ on $\overline{M}\times[0,\infty)$. Thus it follows from the parabolic Harnack inequality for $(\ref{equn_F1n-1F2t3})$ that there exists $C>0$ independent of $j\in \mathbb{N}$ such that
\begin{align}
(F_1+(n-1)F_2)(r,t)\leq C\,(F_1+(n-1)F_2)(r_1,t_j)\leq - C\epsilon,
\end{align}
and hence by $(\ref{inequn_pin1})$,
\begin{align}
F_1(r,t)\leq \frac{-C\,\epsilon}{n},
\end{align}
for $(r,t)\in U_{j,\frac{\delta}{2}}=\{(r,t):\,|r-r_1|\leq \frac{\delta}{2},\,\,t_j+\frac{\delta^2}{8}\leq t\leq t_j+\frac{\delta^2}{4}\}$ and any $t_j\geq \delta^2$. Now we choose a subsequence $\{t_{j_k}\}_{k=1}^\infty\subseteq \{t_j\}_{j=1}^\infty$ such that $t_{j_1}>\delta^2$ and $t_{j_{k+1}}>t_{j_k}+\delta^2$, for $k\geq1$. Therefore,
\begin{align*}
a(r_1,t_{j_p})&=a(r_1,0)(m^{-1}\xi(t_{j_p}))^{\frac{1}{2}}e^{-\int_0^{t_{j_p}}F_1(r_1,\tau)d\tau}\\
&\geq a(r_1,0)(m^{-1}\xi(t_{j_p}))^{\frac{1}{2}}e^{\sum_{k=1}^{p-1}\int_{t_{j_k}+\frac{\delta^2}{8}}^{t_{j_k}+\frac{\delta^2}{4}}\frac{C\,\epsilon}{n}d\tau}\\
&=a(r_1,0)(m^{-1}\xi(t_{j_p}))^{\frac{1}{2}}e^{\frac{\delta^2C\epsilon}{8n}\,(p-1)}\to\infty,
\end{align*}
as $p\to\infty$, contradicting with Lemma \ref{lem_absBD}. Therefore,
\begin{align*}
(F_1+(n-1)F_2)(r_1,t)\to 0
\end{align*}
as $t\to\infty$. Since by Theorem \ref{thm_QPn-1Qunibd2}, $F_1$ and $F_2$ are uniformly bounded in $C^k$-norm  on $\overline{B_{r_1+\delta}}(0)\times[0,\infty)$ for any $k\geq0$, by the Harnack inequality for $(\ref{equn_Pn-1Qt2})$,
\begin{align*}
(F_1+(n-1)F_2)\to 0
\end{align*}
uniformly on $\overline{B_{r_1}}(0)$, as $t\to\infty$. Again, by Shi's interior estimates on the Ricci flow,
\begin{align*}
(F_1+(n-1)F_2)\to 0
\end{align*}
uniformly in $C^k$-norm on $\overline{B_{r_1}}(0)$ for any $k\geq0$, as $t\to\infty$. Moreover, $(m\xi(t)^{-1})^{\frac{1}{2}}a(r,t)$ and $(m\xi(t)^{-1})^{\frac{1}{2}}b(r,t)$ are increasing in $t$, and hence $a(r,t)$ and $b(r,t)$ converge to $a_{\infty}(r)$ and $b_{\infty}(r)$ locally uniformly in $C^k$-norm for any $k\geq0$ on $B_{r_0}(0)$, by $(\ref{equn_aintt0})$, $(\ref{equn_bintt0})$ and Lemma \ref{lem_absBD}. This completes the proof of the theorem.

\end{proof}

\vskip0.2cm

\section{Completeness of the limit metric}\label{section_5}
Let $\overline{M}$ be of dimension $n\geq3$. In this section, we will prove Theorem \ref{thm_main}. To this end, it suffices to prove that the volume of $B_{r_0}(0)$ under the limiting metric $g_{\infty}$ is infinity, as $g_{\infty}=a_{\infty}(r)^2dr^2+b_{\infty}(r)^2g_{\mathbb{S}^{n-1}}$ is rotationally symmetric and locally hyperbolic. We remark that for $m>1$, in comparison to the discussion on the volume growth of $M$ with respect to $g(t)$ for the case when $m=1$ in \cite{Li4}, we have to deal with the asymptotic behavior of the solution to a second order linear ordinary differential inequality as $t\to\infty$, instead of that of a first order differential inequality.

\begin{proof}[Proof of Theorem \ref{thm_main}.]
Since $\xi^{-\frac{1}{2}}a(r,t)$ and $\xi^{-\frac{1}{2}}b(r,t)$ are increasing in $t$ and $\xi(t)\to1$ as $t\to\infty$, it remains to prove that
\begin{align}
\text{Vol}_{g(t)}(B_{r_0}(0))=\omega_n\int_0^{r_0}a(r,t)b(r,t)^{n-1}dr\to\infty\,
\end{align}
as $t\to\infty$, where $\omega_n$ is the area of the unit sphere $\mathbb{S}^{n-1}$ in $\mathbb{R}^n$.

Let $R_g$ be the scalar curvature of $g$, and define
\begin{align*}
R_g^\circ=R_g+n(n-1)=(F_1+(n-1)F_2)+n(n-1)(1-\xi(t)^{-1}).
\end{align*}
By $(\ref{equn_NRF1})$,
\begin{align*}
\frac{d}{dt} \text{Vol}_{g(t)}(B_{r_0}(0))&=\int_{B_{r_0}(0)}-R_{g(t)}^\circ dV_{g(t)}\\
&=\int_{B_{r_0}(0)}-(F_1+(n-1)F_2)-n(n-1)(1-\xi(t)^{-1}) dV_{g(t)}.
\end{align*}
By $(\ref{equn_Pn-1Qt})$,
\begin{align*}
\partial_tR_{g(t)}^\circ=\Delta_gR_{g(t)}^\circ+2(1-n)R_{g(t)}^\circ+2(P^2+(n-1)Q^2)
\end{align*}
on $\overline{B_{r_0}}\times[0,\infty)$. By $(\ref{equn_Pn-1Qb})$, $(\ref{equn_nPn-1Qb})$ and $(\ref{equn_DtHb})$,
\begin{align*}
\frac{\partial}{\partial n_g}R_{g(t)}^\circ=-2\eta'(t)+2[\frac{F_1}{n-1}+(n-2)(b(r_0,t)^{-2}-\frac{\eta(t)^2}{(n-1)^2}+1)+1-\xi(t)^{-1}]\eta(t),
\end{align*}
on $\partial M\times[0,\infty)$, where $n_{g}$ is the outer normal vector field of $\partial M$ with respect to $g(t)$. Therefore,
\begin{align*}
&\frac{d}{dt}\int_{B_{r_0}(0)}R_{g(t)}^\circ dV_{g(t)}\\
=&\int_{B_{r_0}(0)}[-(R_g^\circ)^2+\Delta_gR_{g(t)}^\circ+2(1-n)R_{g(t)}^\circ+2(P^2+(n-1)Q^2)]dV_{g(t)}\\
=&\int_{B_{r_0}(0)}[-(P+(n-1)Q)^2+2(P^2+(n-1)Q^2)+2(1-n)R_{g(t)}^\circ]dV_g+\int_{\partial B_{r_0}(0)}\frac{\partial}{\partial n_g}R_{g(t)}^\circ dS_{g(t)}\\
=&\int_{B_{r_0}(0)}[P^2-2(n-1)PQ-(n-1)(n-3)Q^2+2(1-n)R_{g(t)}^\circ]dV_g\\
&+\int_{\partial B_{r_0}(0)}-2\eta'(t)+2[\frac{F_1}{n-1}+(n-2)(b(r_0,t)^{-2}-\frac{\eta(t)^2}{(n-1)^2}+1)+1-\xi(t)^{-1}]\eta(t) dS_{g(t)}.
\end{align*}
When $m<1$, we have $(1-\xi(t)^{-1})<0$, and hence, by $(\ref{inequn_pin1})$ and the condition on $\eta(t)=H_{\bar{g}(t)}\big|_{\partial M}+\rho(t)$, we obtain
\begin{align}\label{equn_D2VolofM}
&\frac{d}{dt}\int_{M}R_{g(t)}^\circ dV_{g(t)}\\
=&\int_{M}(1-n)(n-3)F_2^2-2(n-1)F_1F_2+F_1^2dV_{g(t)}\nonumber\\
&-2(n-1)(n-2)(1-\xi(t)^{-1})\int_M[F_1+(n-1)F_2]dV_g\nonumber\\
&-n(n-2)(n-1)^2(1-\xi(t)^{-1})^2\text{Vol}_{g(t)}(M)+2(1-n)\int_{M}R_{g(t)}^\circ\,dV_g\nonumber\\
&+\int_{\partial M}-2\eta'(t)+2[\frac{F_1}{n-1}+(n-2)(b(r_0,t)^{-2}-\frac{\eta(t)^2}{(n-1)^2}+1)+1-\xi(t)^{-1}]\eta(t) dS_{g(t)}\nonumber\\
\leq&\,\,\,2(1-n)\int_{M}R_{g(t)}^\circ\,dV_g+f(t),\nonumber
\end{align}
for $t\geq0$, where
\begin{align*}
f(t)&=\int_{\partial M}(n-2)(b(r_0,t)^{-2}-\frac{\eta(t)^2}{(n-1)^2}+1)\eta(t)dS_{g(t)}\\
&=\int_{\partial M}(n-2)(m\xi(t)^{-1}b(r_0,0)^{-2}e^{2\int_0^tF_2(r_0,\tau)d\tau}-\frac{(H_{\bar{g}(t)}\big|_{\partial M}+\rho(t))^2}{(n-1)^2}+1)\eta(t)dS_{g(t)}\\
&\leq\, -\frac{(n-2)}{(n-1)^2}\int_{\partial M}(2\rho(t)\,H_{\bar{g}(t)}\big|_{\partial M}+\rho(t)^2)\eta(t)dS_{g(t)}.
\end{align*}
Since $\rho(0)=0$, $\rho(t)'\geq0$ for $t\geq0$ and $\rho'(t)>0$ for $t\in(0,\epsilon)$, and
\begin{align*}
b(r_0,t)=(m^{-1}\xi(t))^{\frac{1}{2}}b(r_0,0)e^{-\int_0^tF_2(r_0,\tau)d\tau}
\end{align*}
is uniformly bounded from below for $t\geq0$, there exists $\lambda>0$ such that
\begin{align*}
f(t)\leq -\lambda
\end{align*}
for $t\geq \epsilon$ and hence, for $m<1$, we have
\begin{align*}
\frac{d}{dt}\int_{M}R_{g(t)}^\circ dV_{g(t)}\leq 2(1-n)\int_{M}R_{g(t)}^\circ\,dV_g-\lambda
\end{align*}
for $t\geq \epsilon$. Thus, by integration,
\begin{align*}
\int_{M}R_{g(t)}^\circ dV_{g(t)}\leq -\frac{\lambda}{2(n-1)}+e^{-2(n-1)t}[e^{2(n-1)\epsilon}\int_{M}R_{g(\epsilon)}^\circ dV_{g(\epsilon)}+\frac{\lambda}{2(n-1)}e^{2(n-1)\epsilon}],
\end{align*}
for $t\geq\epsilon$. Therefore, for $m\in(0,1)$,
\begin{align*}
\text{Vol}_{g(t)}(M)=\text{Vol}_{g(\epsilon)}(M)-\int_\epsilon^t\int_MR_{g(\tau)}^\circ dV_{g(\tau)}d\tau\to+\infty
\end{align*}
as $t\to\infty$.

Now we consider the case when $m>1$. Define the function $y(t)=\text{Vol}_{g(t)}(M)$, for $t\geq0$. Recall that $\xi(t)\to 1$ as $t\to\infty$. Hence, by the condition satisfied by $\rho(t)$, there exists $\beta>0$ such that
 \begin{align*}
 b(r_0,t)^{-2}-\frac{\eta(t)^2}{(n-1)^2}+1&=m\xi(t)^{-1}b(r_0,0)^{-2}e^{2\int_0^tF_2(r_0,\tau)d\tau}-\frac{(H_{\bar{g}(t)}\big|_{\partial M}+\rho(t))^2}{(n-1)^2}+1\\
 &\leq m\xi(t)^{-1}b(r_0,0)^{-2}-\frac{(H_{\bar{g}(t)}\big|_{\partial M}+\rho(t))^2}{(n-1)^2}+1\\
 &= 1-\xi(t)^{-1}-\frac{1}{(n-1)^2}(2\rho(t)H_{\bar{g}(t)}\big|_{\partial M}+\rho(t)^2)\\
 &\leq -\frac{1}{(n-1)^2}\rho(t)^2\leq -\beta,
 \end{align*}
 for $t$ sufficiently large. Also, by $(\ref{inequn_pin1})$,
 \begin{align*}
 &(1-n)(n-3)F_2^2-2(n-1)F_1F_2+F_1^2-2(n-1)(n-2)(1-\xi(t)^{-1})[F_1+(n-1)F_2]\\
 \leq&-\frac{1}{n^2}(F_1+(n-1)F_2)^2-2(n-1)(n-2)(1-\xi(t)^{-1})[F_1+(n-1)F_2]\\
 \leq&n^2(n-1)^2(n-2)^2(1-\xi(t)^{-1})^2.
 \end{align*}
 Recall that $\eta'(t)\geq0$ for $t\geq0$, and $b(r_0,t)$ is uniformly bounded from below for $t\in[0,\infty)$. Therefore, by $(\ref{inequn_pin1})$ and the first equality in $(\ref{equn_D2VolofM})$, there exists a small constant $\beta_1>0$ and a large constant $N>0$, such that
\begin{align*}
y''(t)\geq &[-n^2(n-1)^2(n-2)^2+n(n-2)(n-1)^2](1-\xi(t)^{-1})^2\text{Vol}_{g(t)}(M)-2(1-n)\int_{M}R_{g(t)}^\circ\,dV_g\\
&+\int_{\partial M}2\eta'(t)-2[\frac{F_1}{n-1}+(n-2)(b(r_0,t)^{-2}-\frac{\eta(t)^2}{(n-1)^2}+1)+1-\xi(t)^{-1}]\eta(t) dS_{g(t)}\\
=&-2(n-1)y'(t)-n(n-1)^2(n-2)(n^2-2n-1)(1-\xi(t)^{-1})^2y(t)\\
&+\int_{\partial M}2\eta'(t)-2[\frac{F_1}{n-1}+(n-2)(b(r_0,t)^{-2}-\frac{\eta(t)^2}{(n-1)^2}+1)+1-\xi(t)^{-1}]\eta(t) dS_{g(t)}\\
\geq&-2(n-1)y'(t)-n(n-1)^2(n-2)(n^2-2n-1)(1-\xi(t)^{-1})^2y(t)+\beta_1,
\end{align*}
for $t\geq N$. Therefore, for any small $\delta>0$, there exists $T=T(\delta)>N$ such that
\begin{align*}
y''(t)\geq -2(n-1)y'(t)-\delta y(t)+\beta_1
\end{align*}
 for $t\geq T$. Let $z(t)=y(t)-\frac{\beta_1}{\delta}$, for $t\geq T$. Hence,
 \begin{align}\label{inequn_D2z}
z''(t)+2(n-1)z'(t)+\delta z(t)\geq0,
\end{align}
 for $t\geq T$. We notice that the eigenvalue equation
  \begin{align*}
\lambda^2+2(n-1)\lambda+\delta =0,
\end{align*}
has two negative eigenvalues when $\delta>0$ is small:
\begin{align*}
\lambda_1=1-n+\sqrt{(n-1)^2-\delta},\,\,\,\,\lambda_2=1-n-\sqrt{(n-1)^2-\delta}.
\end{align*}
Let $z(t)=e^{\lambda_1t}\zeta(t)$ for $t\geq T$, and hence by $(\ref{inequn_D2z})$,
\begin{align*}
\frac{d}{dt}(\zeta(t)'+2(\lambda_1+n-1)\zeta(t))=\zeta(t)''+2(\lambda_1+n-1)\zeta'(t)\geq0
\end{align*}
for $t\geq T$. We now integrate each side of the inequality with respect to $t$ and obtain
\begin{align*}
\zeta(t)'+2(\lambda_1+n-1)\zeta(t)\geq C_1
\end{align*}
for $t\geq T$, with some constant $C_1\in \mathbb{R}$. Hence,
\begin{align*}
\zeta(t)\geq C_1+C_2e^{-2(n-1+\lambda_1)t}
\end{align*}
for $t\geq T$, with some constants $C_1,\,C_2\in \mathbb{R}$. Therefore,
\begin{align*}
z(t)\geq C_1e^{\lambda_1t}+C_2e^{\lambda_2t}
\end{align*}
and hence,
\begin{align*}
y(t)\geq C_1e^{\lambda_1t}+C_2e^{\lambda_2t}+\frac{\beta_1}{\delta}
\end{align*}
for $t\geq T$. By the arbitrary choice of small $\delta>0$, we have
\begin{align*}
\text{Vol}_{g(t)}(M)=y(t)\to+\infty
\end{align*}
as $t\to\infty$. This completes the proof of the theorem.

\end{proof}

\vskip0.2cm

\section{The special case of dimension two}\label{section_6}

For short-time existence of the solution $g(t)$ to $(\ref{equn_NRF1d2})-(\ref{equn_Ninc2d2})$, see the preliminaries. Let $T>0$ be the largest existence time of the solution $g(t)$ to the initial-boundary problem $(\ref{equn_NRF1d2})-(\ref{equn_Ninc2d2})$.

Notice that $F_1=0$ at $t=0$. By the maximum principle for $(\ref{equn_n2Pt})$, as discussed in the proof of Lemma \ref{lem_P_Q-0}, there exists $\epsilon_0\in(0,T)$ such that $F_1$ cannot attain its positive maximum or negative minimum in $B_{r_0}(0)\times(0,\epsilon_1]$ for any $\epsilon_1\in(0,\epsilon_0)$. 
We first show that
\begin{align}\label{inequn_P-0}
F_1(r,t)\leq0
\end{align}
on $\overline{B_{r_0}}(0)\times[0,T)$. After that, we will employ the argument in \cite{Li3} to establish the long-time existence of the solution $g(t)$ and the locally uniform convergence of $g(t)$ to the complete hyperbolic metric in $M$. Notice that the boundary estimates and the variational argument for the asymptotic behavior of the volume of $M$ as $t\to\infty$ for $n\geq3$ do not apply to the case $n=2$, because of the vanishing of the coefficient $(n-2)$ for certain terms in the related formulas.

\begin{thm}\label{thm_F1nonpositiveD2}
For $n=2$, assume the boundary data $\eta$ in $(\ref{equn_Nbdc1d2})$ is defined by
\begin{align*}
\eta(t)=k_{\bar{g}}\big|_{\partial M}+\rho(t)
\end{align*}
for $t\geq0$, with $k_{\bar{g}(t)}\big|_{\partial M}$ defined in $(\ref{equn_Hbd02D})$ and $\rho=\rho(t)\in C^k([0,\infty))$ satisfies the compatibility condition $(\ref{equn_comptmk})$, $\rho'(t)\geq 0$ for $t\geq0$ and $\rho'(t)>0$ for $t\in(0,t_1)$ with some $t_1>0$. Moreover, if $m>1$, we assume that 
$(\ref{inequn_rhodt2})$ holds for $t\geq0$, with $n=2$. Then for the solution $g(t)$ to the initial-boundary problem $(\ref{equn_NRF1d2})-(\ref{equn_Ninc2d2})$, $(\ref{inequn_P-0})$ holds on $\overline{B_{r_0}}(0)\times[0,T)$, and
\begin{align}
F_1<0
\end{align}
on $\overline{M}\times(0,T)$.
\end{thm}
\begin{proof}
The proof is a modification of the proof of Theorem 6.1 in \cite{Li4}. Assume the contrary, there exists $t_0\in[0,T)$, such that $(\ref{inequn_P-0})$ holds for $t\in [0,t_0]$ and for any $\delta>0$, there exists $t_1\in(t_0,t_0+\delta)$ such that
\begin{align*}
F_1(r_0,t_1)>0,
\end{align*}
where we have applied the maximum principle for $(\ref{equn_n2Pt})$, as in the proof of Lemma \ref{lem_P_Q-0}.

Now we define a rotationally symmetric function $\tau(x)\in C^2(\overline{M})$ such that $\tau(x)=\tau(r(x))\geq \frac{1}{2}$ for $x\in\overline{M}\subseteq \mathbb{R}^2$, and $\tau(x)=1-s(r_0,t_0)+s(r(x),t_0)$ in a small neighborhood of $\partial M$, where $r(x)=|x|$ is the distance from $x$ to the origin in the Euclidean space $\mathbb{R}^2$. Hence, $\tau(x)=1$ for $x\in \partial M$, and
\begin{align}\label{equn_taunb}
\frac{\partial\tau}{\partial n_{g(t_0)}}=1
\end{align}
on $\partial B_{r_0}(0)$, where $n_{g(t_0)}$ is the outer normal vector field of $\partial B_{r_0}(0)$ with respect to $g(t_0)$. Let $\theta=e^{-at}F_1e^{-N\tau(x)}$, with $a,\,N>0$ two large constants to be determined. Thus for any $\delta>0$, by $(\ref{equn_n2Pt})$, we obtain
\begin{align}\label{equn_thetat}
\theta_t'&=e^{-at}\Delta_gF_1\,e^{-N\tau(x)}+(-a+2F_1+2-4\xi(t)^{-1})\theta\\
&=\Delta_g\theta+2Ne^{-N\tau(x)}\nabla_g(e^{-at}F_1)\cdot\nabla_g\tau(x)-e^{-at}F_1\Delta_g(e^{-N\tau(x)})+(-a+2F_1+2-4\xi(t)^{-1})\theta\nonumber\\
&=\Delta_g\theta+2N\nabla_g\tau(x)\cdot \nabla_g\theta+[N^2 \big|\nabla_g\tau(x)\big|_g^2+N \Delta_g\tau(x)-a+2F_1+2-4\xi(t)^{-1}]\theta\nonumber
\end{align}
on $\overline{M}\times[t_0,t_0+\delta]$. Also, by $(\ref{equn_Dtkb})$, we obtain
\begin{align}\label{equn_thetanb}
\frac{\partial}{\partial n_g}\theta&=(\eta\,-\,N\,\frac{\partial\tau(x)}{\partial n_g})\theta+e^{-at-N\tau(x)}[(1-\xi(t)^{-1})\eta(t)-\eta'(t)]\\
&=(\eta\,-\,N\,\frac{\partial\tau(x)}{\partial n_g})\theta+e^{-at-N\tau(x)}[(1-\xi(t)^{-1})\rho(t)-\rho'(t)]\nonumber
\end{align}
on $\partial M\times[t_0,t_0+\delta]$, where $g=g(t)$ and $n_g$ is the outer normal vector field of $\partial M$ with respect to $g(t)$. By $(\ref{equn_taunb})$ and continuity of $g(t)$, there exists $\delta>0$ small such that
\begin{align}
\frac{9}{10}\leq \frac{\partial\tau(x)}{\partial n_g} \leq \frac{11}{10}
\end{align}
on $\partial M\times[t_0,t_0+\delta]$, and hence, we now choose $N>0$ large so that
\begin{align}\label{inequn_taunb2}
\eta\,-\,N\,\frac{\partial\tau(x)}{\partial n_g}<0
\end{align}
on $\partial M\times[t_0,t_0+\delta]$. Next, we take $a>0$ large so that
\begin{align*}
N^2 \big|\nabla_g\tau(x)\big|_g^2+N \Delta_g\tau(x)-a+2F_1+2-4\xi(t)^{-1}<0
\end{align*}
on $\overline{M}\times[t_0,t_0+\delta]$. By the assumption, $\theta(x,t_0)\leq0$ for $x\in \overline{M}$, and there exists $t_1\in (t_0,t_0+\delta)$ such that $\sup_{\overline{M}}\theta(\cdot,t_1)>0$. Therefore, by the maximum principle for $(\ref{equn_thetat})$, there exists $t_2\in(t_0,t_1]$ such that
\begin{align*}
\theta(r_0,t_2)=\sup_{\overline{M}\times[t_0,t_2]}\theta>0,
\end{align*}
and hence,
\begin{align*}
\frac{\partial}{\partial n_g}\theta(r_0,t_2)\geq0,
\end{align*}
contradicting with $(\ref{equn_thetanb})$, by $(\ref{inequn_taunb2})$ and $(\ref{inequn_rhodt2})$. Therefore, the inequality $(\ref{inequn_P-0})$ holds on $\overline{M}\times[0,T)$.

Now if $\rho'(t)>0$ for $t\in(0,t_1)$, by the strong maximum principle for $(\ref{equn_n2Pt})$, $F_1<0$ on $M\times(0,T)$. Assume $F_1(r_0,t_1)=0$ for some $t_1\in(0,T)$. Thus, by the Hopf Lemma(see Lemma 2.8 on Page 12 in \cite{Liebm}), one gets
\begin{align*}
\frac{\partial}{\partial n_g}F_1(r_0,t_1)>0,
\end{align*}
contradicting with $(\ref{equn_Dtkb})$ and $(\ref{inequn_rhodt2})$. This completes the proof of the theorem.



\end{proof}

We next show the long-time existence of the solution $u$ to the initial-boundary value problem $(\ref{equn_ut})-(\ref{equn_uint})$. As in \cite{Li4}, recall the comparison theorem derived in \cite{Li3}.
\begin{thm}(Theorem 3.1 and Remark 3.1 in \cite{Li3}) Let $u_1$ be a subsolution to $(\ref{equn_ut})$ satisfying
\begin{align*}
&(u_1)_t'\leq e^{-2u_1}(\Delta_{g(0)}u_1-K_{g(0)})-1,\,\,\text{in}\,\,\overline{M}\times[0,T),\\
&\frac{\partial u_1}{\partial n_{g(0)}}+k_{g(0)}\leq\eta_1 e^{u_1},\,\,\text{on}\,\,\partial M\times[0,T),\\
&u_1\big|_{t=0}\leq u_{01},\,\,\,\text{in}\,\,\overline{M},
\end{align*}
and $u_2$ be a supersolution to $(\ref{equn_ut})$ satisfying
\begin{align*}
&(u_2)'_t\geq e^{-2u_2}(\Delta_{g(0)}u_2-K_{g(0)})-1,\,\,\text{in}\,\,\overline{M}\times[0,T),\\
&\frac{\partial u_2}{\partial n_{g(0)}}+k_{g(0)}\geq\eta_2 e^{u_2},\,\,\text{on}\,\,\partial M\times[0,T),\\
&u_2\big|_{t=0}\geq u_{02},\,\,\,\text{in}\,\,\overline{M},
\end{align*}
with $u_{01}\leq u_{02}$ in $\overline{M}$ and $\eta_1\leq \eta_2$ on $\partial M\times[0,T)$. Then we have
\begin{align}
u_1\leq u_2
\end{align}
for $(x,t)\in \overline{M}\times[0,T)$.
\end{thm}
We restate the conclusion regarding the upper bound control of the solution from Theorem 1.3 in \cite{Li3} as the following theorem:
\begin{thm}\label{thm_upperbdun2}
Let $(\overline{M},g)$ be a compact surface with its interior $M$ and boundary $\partial M$, with $g$ a fixed Riemannian metric on $\overline{M}$. Consider the initial-boundary value problem
\begin{align}
&\label{equn_ut2}u_t=e^{-2u}(\Delta_gu-K_g)-1,\,\,\text{in}\,\,\overline{M}\times[0,\infty),\\
&\label{equn_unb2}\frac{\partial u}{\partial n_g}+k_{g}=\eta e^u,\,\,\text{on}\,\,\partial M\times[0,\infty),\\
&u\label{equn_uint2}\big|_{t=0}=u_0,\,\,\,\text{in}\,\,\overline{M},
\end{align}
We assume that $u_0\in C^{2,\alpha}(\overline{M})$ and $\eta\in C^{1+\alpha,\frac{1}{2}+\frac{\alpha}{2}}(\partial M\times [0,T_1])$ for all $T_1>0$, and also the compatibility condition holds on $\partial M\times\{0\}$:
\begin{align*}
\frac{\partial u_0}{\partial n_{g}}+k_{g}=\eta(\cdot,0)e^{u_0}.
\end{align*}
Moreover, assume that $(\ref{inequn_etaupperbd})$ holds on $\partial M\times[0,\infty)$, where $y(t)\in C^3([0,\infty))$ is some positive function satisfying $(\ref{inequn_ODEgrowth})$ for $t\in[0,\infty)$. Then for any $T_1>0$, there exists $C=C(T_1)>0$ such that the solution $u$ to $(\ref{equn_ut2})-(\ref{equn_uint2})$ satisfies
\begin{align}\label{inequn_upperbdun2}
u\leq C(T_1)
\end{align}
on $\overline{M}\times[0,T_1)$.
\end{thm}

Now we show the long-time existence of the solution $g(t)$ to the initial-boundary value problem $(\ref{equn_NRF1d2})-(\ref{equn_Ninc2d2})$.
\begin{thm}\label{thm_existlongtimen2}
Assume $n=2$, $k\geq 2$. Let $\eta(t),\,\rho\in C^k([0,\infty))$ be defined as in Theorem \ref{thm_F1nonpositiveD2}. Moreover, $(\ref{inequn_etaupperbd})$ holds on $\partial M\times[0,\infty)$, where $y(t)\in C^3([0,\infty))$ is some positive function satisfying $(\ref{inequn_ODEgrowth})$ for $t\in[0,\infty)$.  Then there exists a unique solution $g(t)$ to $(\ref{equn_NRF1d2})-(\ref{equn_Ninc2d2})$ on $\overline{M}\times[0,\infty)$.
\end{thm}
\begin{proof}
By $(\ref{inequn_P-0})$, the solution
\begin{align*}
u(r,t)=-\int_0^tP(r,\tau)d\tau=-\int_0^tF_1(r,\tau)d\tau+\frac{1}{2}(\ln(\xi(t))-\ln(m))
 \end{align*}
 to the initial-boundary value problem $(\ref{equn_ut})-(\ref{equn_uint})$ is increasing in $t$ if $m\in(0,1)$, for $(r,t)\in[0,r_0]\times[0,T)$; while $u(r,t)\geq -\ln(m)$ for $m>1$, for $(r,t)\in[0,r_0]\times[0,T)$. Hence, $u(r,t)$ is uniformly bounded from below on $[0,r_0]\times[0,T)$ for any $T>0$.

 When $m\in(0,1)$, we have $u(r,t)\geq 0$ for $(r,t)\in[0,r_0]\times[0,T)$. That is to say, the function $v=u_t'(r,t)=-P(r,\tau)\geq0$, for $(r,t)\in[0,r_0]\times[0,T)$. Therefore, $u(\cdot,t)$ is a subsolution to
\begin{align}\label{equn_statlim}
\Delta_{g(0)}u-K_{g(0)}-e^{2u}=0
\end{align}
in $\overline{M}$ for $t\geq0$. Differentiating both sides of $(\ref{equn_ut})$ with respect to $t$ yields
\begin{align}\label{equn_vt}
v_t=e^{-2u}\Delta_{g(0)}v-2(v^2+v)
\end{align}
on $\overline{M}\times[0,T)$.

On the other hand, by Theorem \ref{thm_upperbdun2}, for any $T_1>0$, $(\ref{inequn_upperbdun2})$ holds on $\overline{M}\times[0,T_1)$.

In summary, if the largest existence time $T$ of the solution $g(t)$ satisfies $T<\infty$, then $u$ is uniformly bounded on $\overline{M}\times[0,T)$. Hence, by the standard regularity theory for parabolic equations, there exists a positive constant $C=C(T)>0$ such that
\begin{align*}
\|u\|_{C^{4+\alpha,2+\frac{\alpha}{2}}(M\times[0,T))}\leq C(T)
\end{align*}
for any $\alpha\in(0,1)$. Therefore, the solution $u$ can be extended to time $T$, and hence, the solution $g(t)$ exists for all $t\geq0$. This proves the theorem.

\end{proof}

Now we are ready to show the convergence of the solution $g(t)$ to $(\ref{equn_NRF1d2})-(\ref{equn_Ninc2d2})$ to a complete hyperbolic metric in $M$.

\begin{proof}[Proof of Theorem \ref{thm_convern2}.]
We first consider the case when $m\in(0,1)$. Since $u(\cdot,t)\in C^2(\overline{M})$ is a subsolution to $(\ref{equn_statlim})$, by the comparison theorem,
\begin{align*}
u(x,t)\leq u_{LN}(x)
\end{align*}
for $(x,t)\in M\times[0,\infty)$, where $u_{LN}$ is the solution to the Loewner-Nirenberg problem of $(\ref{equn_statlim})$ on $M$. Thus, by applying the standard regularity theory for parabolic equations to $(\ref{equn_ut})$, for each compact subset $F\subseteq M$, there exists a constant $C=C(F)>0$ such that
 \begin{align*}
 \|u\|_{C^{4,2}(F\times[0,\infty))}\leq C(F).
 \end{align*}
 Since $u(x,t)$ is increasing in $t$, $u(x,t)$ converges to $u_{\infty}(x)\leq u_{LN}(x)$ pointwisely in $M$ as $t\to\infty$. Therefore, by the Harnack inequality for $(\ref{equn_vt})$, $v=u_t\to0$ and $u(\cdot,t)\to u_{\infty}(\cdot)$ locally uniformly in $M$ as $t\to\infty$. Hence, by applying standard regularity theory to $(\ref{equn_ut})$, $u(x,t)\to u_{\infty}(x)$ locally uniformly in $C^{2,\alpha}$ in $M$. Therefore, $u_{\infty}$ satisfies $(\ref{equn_statlim})$ in $M$.

 For $m>1$, let $\overline{\Omega_1}$ be any smooth compact sub-domain in $M$, with its interior $\Omega_1$, and $U$ the solution to the Loewner-Nirenberg problem on $\Omega_1$. Let $\phi(t)=\sup_{\Omega_1}(u(\cdot,t)-U(\cdot))$ for $t\geq0$. By the argument in the interior upper bound estimates in Section 4 of \cite{Li3}, using Hamiton's trick on the ordinary differential inequality satisfied by $\phi(t)$ for $t\in[0,\infty)$, we have
 \begin{align*}
 \limsup_{t\to\infty}u(p,t)\leq U(p)
 \end{align*}
 for $p\in \Omega_1$. Therefore, $u(\cdot,t)$ is locally uniformly bounded from above for $t\in[0,\infty)$ in $M$. Since $F_1\leq0$ on $\overline{M}\times[0,\infty)$, we obtain
 \begin{align}\label{equn_u2D3}
-\int_0^tF_1(r,\tau)d\tau=u(r,t)-\frac{1}{2}(\ln(\xi(t))-\ln(m))
 \end{align}
is increasing in $t$, and hence, $|u|$ is uniformly bounded from below on any given compact subset $\overline{\Omega_1}\subseteq M$ for $t\geq0$, and also,
 \begin{align*}
u(r,t)\to u_{\infty}(r)
 \end{align*}
as $t\to\infty$, for $r\in[0,r_0)$. Hence, by standard interior estimates, we derive from $(\ref{equn_ut})$ that for any $j\in \mathbb{N}$, $\alpha\in(0,1)$ and any smooth compact sub-domain $\overline{\Omega_1}\subseteq M$, there exists $C=C(j,\alpha,\overline{\Omega_1})>0$ such that
\begin{align}\label{inequn_ubounds2D3}
\|u\|_{C^{2j+\alpha,j+\frac{\alpha}{2}}(\overline{\Omega_1}\times[T,T+1])}\leq C,
\end{align}
for any $T\geq0$. Let $f(r,t)=-\int_0^tF_1(r,\tau)d\tau$ and
\begin{align*}
\varphi(r,t)=f_t'(r,t)=-F_1(r,t)
 \end{align*}
 for $(r,t)\in[0,r_0]\times[0,\infty)$. By Theorem \ref{thm_F1nonpositiveD2}, $\varphi(r,t)\geq0$ for $(r,t)\in[0,r_0]\times[0,\infty)$. By $(\ref{equn_u2D3})$ and $(\ref{inequn_ubounds2D3})$, for any $j\in \mathbb{N}$, $\alpha\in(0,1)$ and any smooth compact sub-domain $\overline{\Omega_1}\subseteq M$, there exists $C_1=C_1(j,\alpha,\overline{\Omega_1})>0$ such that
\begin{align}\label{inequn_ubounds2D3-1}
&\|f\|_{C^{2j+\alpha,j+\frac{\alpha}{2}}(\overline{\Omega_1}\times[T,T+1])}\leq C_1,\\
&\|\varphi\|_{C^{2j-2+\alpha,j-1+\frac{\alpha}{2}}(\overline{\Omega_1}\times[T,T+1])}\leq C_1,
\end{align}
for any $T\geq0$. By $(\ref{equn_vt})$, we have
\begin{align}\label{equn_varphiDt}
\varphi_t'=e^{-2u}\Delta_{g(0)}\varphi-2[\varphi+(2\xi(t)^{-1}-1)]\varphi
\end{align}
on $\overline{M}\times[0,\infty)$. Therefore, by the Harnack inequality for $(\ref{equn_varphiDt})$, $\varphi(\cdot,t)\to0$ and $u(\cdot,t)\to u_{\infty}(\cdot)$ locally uniformly in $M$ as $t\to\infty$. Hence, by the standard interior estimates for $(\ref{equn_ut})$, $u(x,t)\to u_{\infty}(x)$ locally uniformly in $C^{2,\alpha}$ in $M$ and hence, $u_{\infty}$ satisfies $(\ref{equn_statlim})$ in $M$.

It remains to show the completeness of the metric $g_{\infty}=e^{2u_{\infty}}g(0)$ on $M$. Since
\begin{align*}
\text{Vol}_{g(t)}(M)=2\pi\int_0^{r_0}a(r,0)b(r,0)e^{-2\int_0^tP(r,\tau)d\tau}dr=2\pi\,m^{-1}\xi(t)\int_0^{r_0}a(r,0)b(r,0)e^{-2\int_0^tF_1(r,\tau)d\tau}dr,
\end{align*}
with $F_1\leq0$ and $\xi(t)\to1$ as $t\to\infty$, it suffices to show that the volume of $M$ with respect to $g(t)$ goes to infinity, as $t\to\infty$.

By $(\ref{equn_NRFa})$ and $(\ref{equn_NRFb})$,
\begin{align*}
\frac{d}{dt}Vol_g(B_{r_0}(0))=\int_{B_{r_0}(0)}-2P dV_g,
\end{align*}
for $t\geq0$. Taking derivative with respect to $t$ again, for $m\neq1$, by $(\ref{equn_n2Pt})$ and $(\ref{equn_Dtkb})$ one has
\begin{align}
\frac{d}{dt}\int_{B_{r_0}(0)}-2P dV_g
&=\int_{B_{r_0}(0)}-2\Delta_gP+4P dV_g\\
&=-2\int_{\partial M}\frac{\partial P}{\partial n_g}dS_g+4\int_{B_{r_0}(0)}P dV_g\nonumber\\
&=2\int_{\partial M}\eta'(t)-\eta(t)(F_1+1-\xi(t)^{-1}) dS_g+4\int_{B_{r_0}(0)}P dV_g\nonumber\\
&\geq 2\int_{\partial M}\eta'(t)-\eta(t)(1-\xi(t)^{-1}) dS_g+4\int_{B_{r_0}(0)}P dV_g\nonumber\\
&=4\pi b(r_0,t)[\rho'(t)-\rho(t)(1-\xi(t)^{-1})]+4\int_{B_{r_0}(0)}P dV_g\nonumber\\
&\geq 4\pi b(r_0,0)(m^{-1}\xi(t))^{\frac{1}{2}}[\rho'(t)-\rho(t)(1-\xi(t)^{-1})]+4\int_{B_{r_0}(0)}P dV_g\nonumber,
\end{align}
where for the last inequality we have used $(\ref{inequn_rhodt2})$ with $n=2$ when $m>1$, and hence, by $(\ref{inequn_rhoinfty})$ and $(\ref{inequn_rhoinfty2})$,
\begin{align*}
\int_{B_{r_0}(0)}-2P(\cdot,t)dV_g\geq e^{2t_0-2t}\int_{B_{r_0}(0)}-2P(\cdot,t_0)dV_{g(t_0)}+Ce^{-2t}\int_{t_0}^te^{2\tau}\epsilon (1+\tau)^{-1}(\ln(1+\tau))^{-1} d\tau,
\end{align*}
for some uniform constant $C>0$ and $t\geq t_0>T$. Therefore,
\begin{align*}
Vol_{g(t)}(B_{r_0}(0))&= Vol_{g(t_0)}(B_{r_0}(0))+\int_{t_0}^t\int_{B_{r_0}(0)}-2P(\cdot,\tau) dV_{g(\tau)}d\tau\\
&\geq C\epsilon\int_{t_0}^te^{-2\tau}\int_{t_0}^{\tau}e^{2y}(1+y)^{-1}(\ln(1+y))^{-1}dy d\tau-C_2\to\infty
\end{align*}
as $t\to\infty$, where $C_2>0$ is some constant. This completes the proof of the theorem.

\end{proof}

\begin{appendix}
\section{A proof of Theorem \ref{thm_convern2} without the conditions $(\ref{inequn_rhoinfty})$ and $(\ref{inequn_rhoinfty2})$}

By the first part of the proof of Theorem \ref{thm_convern2} in Section \ref{section_6}, it suffices to prove that the limit metric $g_{\infty}$ is complete.

\subsection{Rescaling and the curvature maximum principle}

Let \(h(t)=\xi(t)^{-1}g(t)\). Since
\(K_g=\xi^{-1}K_h\) and \(\xi_t'=2(1-\xi)\), differentiation gives
\begin{align}
 \partial_t h
 &= -2(\xi^{-1}K_h+1)h-\frac{\xi_t'}{\xi}h
  =-\frac{2}{\xi}(K_h+1)h. \label{eq:h-flow}
\end{align}
Also \(h(0)=g_{-1}\), and the boundary curvature of \(h\) is
\begin{equation}\label{eq:kappa}
 \kappa(t)=\sqrt{\xi(t)}\,\eta(t)=k_{g_{-1}}+\sqrt{\xi(t)}\,\rho(t).
\end{equation}
The identity \((\sqrt{\xi})_t'/\sqrt{\xi}=\xi^{-1}-1\) gives
\[
 \kappa_t'=\sqrt{\xi}\bigl[\rho_t'-(1-\xi^{-1})\rho\bigr]\ge0.
\]
For \(m\le1\) this follows from \(\xi\le1\) and
\(\rho,\rho'\ge0\). For \(m>1\) it follows from the extra hypothesis.

The flow is conformal, so \(h=e^{2v}g_{-1}\) and hence $e^{2v}=m\xi(t)^{-1}e^{2u}$, with \(v(\cdot,0)=0\).
The curvature formulas under a conformal change are
\begin{equation}\label{eq:conformal}
 K_{e^{2v}g_{-1}}=e^{-2v}(-1-\Delta_{g_{-1}}v),\qquad
 k_{e^{2v}g_{-1}}=e^{-v}(k_{g_{-1}}+\partial_{\nu_{g_{-1}}}v)
 \quad\text{on }\partial M .
\end{equation}
Thus the local problem for \(v\) is the scalar, strictly parabolic
problem
\begin{equation}\label{eq:v-flow}
 v_t'=\xi^{-1}\bigl[e^{-2v}(\Delta_{g_{-1}}v+1)-1\bigr],\qquad
 \partial_{\nu_{g_{-1}}}v=\kappa(t)e^v-k_{g_{-1}} .
\end{equation}
The initial boundary data agree
because \(\kappa(0)=k_{g_{-1}}\). The compatibility conditions for
differentiating in time follow from the vanishing derivatives of
\(\rho\), since \(v\equiv0\) solves the equations for \(\rho\equiv0\).
Uniqueness and invariance of the data under rotations imply that
\(v\) is radial.

Set \(F=K_h+1\). To compute its evolution, for a variation
\(\partial_t h=2\sigma h\), the identities in \eqref{eq:conformal},
applied with the metric at the current time as the reference metric,
give
\[
 \partial_tK_h=-\Delta_h\sigma-2\sigma K_h,\qquad
 \partial_t k_h=\partial_{\nu_h}\sigma-\sigma k_h .
\]
Substituting \(\sigma=-\xi^{-1}F\) yields
\begin{align}
 F_t'&=\xi^{-1}\bigl(\Delta_hF+2(F-1)F\bigr),\label{eq:F-flow}\\
 \partial_{\nu_h}F-\kappa F&=-\xi\kappa_t'\le0
       \quad\text{on }\partial M,\qquad F(\cdot,0)=0. \label{eq:F-boundary}
\end{align}

Here is the maximum principle argument with the Robin coefficient
made explicit. Fix any compact time interval \([0,T]\) in the
existence interval. Choose a smooth function \(\psi\) on
\(\overline M\) with \(\partial_{\nu_{g_{-1}}}\psi=1\) on the
boundary. For example, choose \(\psi=r-r_0\) in a boundary collar
and extend it smoothly into the disk. Then
\(\partial_{\nu_h}\psi=e^{-v}>0\) on the boundary and has a positive
lower bound on this interval. Choose \(N>0\) such that
\(\kappa-N\partial_{\nu_h}\psi<0\) on the lateral boundary.
For \(q=e^{-ct-N\psi}F\) with some constant $c>0$, equations
\eqref{eq:F-flow}--\eqref{eq:F-boundary} give
\begin{align*}
 q_t'={}&\xi^{-1}\Delta_h q+
       \frac{2N}{\xi}\langle\nabla_h\psi,\nabla_hq\rangle_h\\
     &+\left[
       \frac{N\Delta_h\psi+N^2|\nabla_h\psi|_h^2+2(F-1)}{\xi}
       -c\right]q,\\
 \partial_{\nu_h}q={}&
   (\kappa-N\partial_{\nu_h}\psi)q
    -e^{-ct-N\psi}\xi\kappa_t'.
\end{align*}
Choose \(c\) so that the coefficient of \(q\) in the interior is
strictly negative, which is possible by compactness. At a positive
maximum of \(q\) at time \(t>0\) in the spatial interior one has
\(q_t\ge0\), \(\nabla_hq=0\), and \(\Delta_hq\le0\), contradicting
the interior equation. At a positive boundary maximum one has
\(\partial_{\nu_h}q\ge0\), contradicting the boundary equation.
The initial value is zero, so \(F\le0\) throughout \([0,T]\).
As \(T\) is arbitrary, \eqref{eq:h-flow} now gives
\begin{equation}\label{eq:monotonicity}
 K_h\le-1,\qquad v_t=-\xi^{-1}F\ge0,\qquad h(t)\ge g_{-1}.
\end{equation}

\subsection{Two radial area bounds and the area evolution}

Fix \(t\) in the existence interval. Define
\[
 s(r,t)=\int_0^r e^{v(u,t)}\,d u,\qquad
 S(t)=s(r_0,t),\qquad
 b(s(r,t),t)=e^{v(r,t)}\sinh r.
\]
Then \(h=d s^2+b^2\,d\theta^2\) at this time. Smoothness at the
center gives \(b(0,t)=0\) and \(b_s(0,t)=1\). For \(s>0\) the only
nonzero Christoffel symbols in these coordinates, up to symmetry
of the lower indices, are
\(\Gamma^s_{\theta\theta}=-bb_s\) and
\(\Gamma^\theta_{s\theta}=b_s/b\).
They give \(K_h=-b_{ss}/b\), and the geodesic curvature convention
above gives \(k=b_s/b\) on a circle of constant \(s\).

Write \(B(t)=b(S(t),t)\) and \(A(t)=\text{Area}_{h(t)}(M)\). By
\eqref{eq:monotonicity}, \(b_{ss}\ge b\), so \(b_s\ge1\) and
\[
 \frac{d}{d s}(b_s^2-b^2)=2b_s(b_{ss}-b)\ge0 .
\]
At the center \(b_s^2-b^2=1\), so \(b_s\ge\sqrt{1+b^2}\).
Moreover, \(b_s\le b_s(S,t)=\kappa B\) since \(b_s\) is increasing.
Changing variables from \(s\) to \(b\) therefore gives both bounds
\begin{equation}\label{eq:area-bounds}
 \frac{\pi B}{\kappa}
 \le A
 =2\pi\int_0^S b(s,t)\,d s
 \le 2\pi\bigl(\sqrt{1+B^2}-1\bigr).
\end{equation}
Indeed, the lower bound comes from integrating
\(2\pi b/(\kappa B)\) over \(0\le b\le B\), and the upper bound
comes from integrating \(2\pi b/\sqrt{1+b^2}\) over the same interval.
The metric comparison in \eqref{eq:monotonicity} gives
\(B\ge \sinh r_0\). In particular,
\begin{equation}\label{eq:B-bound}
 \frac{\sqrt{1+B^2}}{B}=\sqrt{1+B^{-2}}
 \le\sqrt{1+\sinh^{-2}r_0}=k_{g_{-1}}.
\end{equation}

In two dimensions
\(\partial_t\,d A_h=\tfrac12\operatorname{tr}_h(\partial_t h)\,d A_h
=-2\xi^{-1}(K_h+1)\,d A_h\).
The boundary has length \(2\pi B\) and constant curvature
\(\kappa(t)\). Hence the Gauss--Bonnet formula gives the exact identity
\begin{equation}\label{eq:area-identity}
 A'=-\frac{2}{\xi}\int_{\overline M}(K_h+1)\,\,d A_h
   =\frac{2}{\xi}\bigl(2\pi\kappa B-2\pi-A\bigr).
\end{equation}
Using the upper bound in \eqref{eq:area-bounds}, then
\eqref{eq:B-bound} and \eqref{eq:kappa}, gives
\begin{align}
 A'
 &\ge \frac{4\pi}{\xi}\bigl(\kappa B-\sqrt{1+B^2}\bigr)
 \ge \frac{4\pi B}{\xi}(\kappa-k_{g_{-1}})
 = \frac{4\pi B}{\sqrt{\xi}}\rho(t). \label{eq:area-growth}
\end{align}
Take \(t_*\in(0,t_1)\), then \(\rho(t_*)>0\) and
\(\rho(t)\ge\rho(t_*)\) for \(t\ge t_*\). Since
\(\xi(t)\le\max\{1,m\}\), the bound becomes
\begin{equation}\label{eq:linear}
 A(t)\ge A(t_*)+
 \frac{4\pi\sinh r_0\,\rho(t_*)}{\sqrt{\max\{1,m\}}}(t-t_*)
 \qquad(t\ge t_*).
\end{equation}
Therefore, $A(t)\to\infty$ as $t\to\infty$. This completes the proof of the theorem.

\end{appendix}

 \end{document}